\documentclass[11pt]{article}

\usepackage{amsmath}
\usepackage{aliascnt}
\usepackage{amsfonts}
\usepackage{amsthm}
\usepackage{amssymb}
\usepackage{dsfont}
\usepackage{graphicx}
\usepackage[hidelinks]{hyperref}
\usepackage[title,titletoc,toc]{appendix}
\usepackage[a4paper,bindingoffset=0.5cm,left=2.3cm,right=2.3cm,
	top=2.5cm,bottom=2cm,footskip=.8cm]{geometry}

\newtheorem{theorem}{Theorem}[section]

\newaliascnt{lemma}{theorem}
\newtheorem{lemma}[lemma]{Lemma}
\aliascntresetthe{lemma}

\newaliascnt{corollary}{theorem}

\aliascntresetthe{corollary}

\newaliascnt{proposition}{theorem}
\newtheorem{proposition}[proposition]{Proposition}
\aliascntresetthe{proposition}

\theoremstyle{definition}

\newaliascnt{definition}{theorem}
\newtheorem{definition}[definition]{Definition}
\aliascntresetthe{definition}

\newaliascnt{remark}{theorem}
\newtheorem{remark}[remark]{Remark}
\aliascntresetthe{remark}

\newcommand{\norm}[1]{\lvert#1\rvert}
\newcommand{\normbig}[1]{\bigl\lvert#1\bigr\rvert}
\newcommand{\abs}[1]{\lvert#1\rvert}
\newcommand{\absbig}[1]{\bigl\lvert#1\bigr\rvert}
\newcommand{\floor}[1]{\lfloor#1\rfloor}

\newcommand{\bmax}{\vee}
\newcommand{\bmin}{\wedge}

\newcommand{\1}{\mathds{1}}
\newcommand{\one}[1]{\mathds{1}_{\{#1\}}}

\newcommand*{\diff}{\mathop{}\!\mathrm{d}}

\newcommand{\R}{\mathbb{R}}
\newcommand{\N}{\mathbb{N}}
\newcommand{\Z}{\mathbb{Z}}
\newcommand{\PP}{\mathbb{P}}
\newcommand{\E}{\mathbb{E}}

\newcommand{\cF}{\mathcal{F}}
\newcommand{\cS}{\mathcal{S}}
\newcommand{\cK}{\mathcal{K}}

\newcommand{\bQ}{\bar{Q}}
\newcommand{\bR}{\bar{R}}
\newcommand{\bI}{\bar{I}}
\newcommand{\bT}{\bar{T}}

\author{Pieter Jacob Storm, Wouter Kager, Michel Mandjes and Sem Borst}
\title{Stability of a Stochastic Ring Network}
\date{\today}

\begin{document}

\maketitle

\begin{abstract}
	In this paper we establish a necessary and sufficient stability condition 
	for a stochastic ring network. Such networks naturally appear in a variety 
	of applications within communication, computer, and road traffic systems. 
	They typically involve \textit{multiple customer types} and some form of 
	\textit{priority structure} to decide which customer receives service. 
	These two system features tend to complicate the issue of identifying a 
	stability condition, but we demonstrate how the ring topology can be 
	leveraged to solve the problem. 

	\medskip
	\noindent
	{\sc Keywords.} {Cellular automata $\diamond$ Communication networks 
	$\diamond$ Fluid models $\diamond$ Ring-topology queueing networks 
	$\diamond$ Stability and bottleneck analysis $\diamond$ Traffic flow 
	theory}

	\medskip
	\noindent
	{\sc Affiliations.} 

	Pieter Jacob Storm\textsuperscript{\textdagger}: Department of Mathematics 
	and Computer Science, Eindhoven University of Technology, Eindhoven, The 
	Netherlands.

	Wouter Kager: Department of Mathematics, Vrije Universiteit Amsterdam, 
	Amsterdam, The Netherlands.

	Michel Mandjes\textsuperscript{\textdagger}: Korteweg--de Vries Institute 
	for Mathematics, University of Amsterdam, Amsterdam,
	The Netherlands; {\sc Eurandom}, Eindhoven University of Technology, 
	Eindhoven, The Netherlands; Amsterdam Business School, Faculty of 
	Economics and Business, University of Amsterdam, Amsterdam, The 
	Netherlands.

	Sem Borst\textsuperscript{\textdagger}: Department of Mathematics and 
	Computer Science, Eindhoven University of Technology, Eindhoven, The 
	Netherlands.

	\medskip
	\textsuperscript{\textdagger}Partly funded by NWO Gravitation project
	{\sc Networks}, grant number 024.002.003.
\end{abstract}

\section{Introduction}
\label{sec: intro}

This paper deals with the analysis of a hybrid cellular automaton and queueing 
network model with a ring structure. Specifically, customers (particles) are 
routed probabilistically between stations (cells) according to a network with 
a ring topology, and customers within the ring have priority over exogenously 
arriving customers. Application domains of such networks include 
packet-switched \textit{optical ring networks}~\cite{herzog2004, yang2004, 
yuang2010}, road traffic intersections in the form of 
\textit{roundabouts}~\cite{belz2016, storm2020}, \textit{multiprocessor 
systems} with ring-based nanophotonic on-chip networks~\cite{bourduas2011, 
Ghosh2019}, and \textit{medium access control}~(MAC) protocols for Local Area 
Networks~(LANs)~\cite{vArem1990, coffman1998}. 

The model that we consider was proposed as a discrete-time Markov chain 
in~\cite{storm2020}. The main contribution of this paper is a necessary and 
sufficient condition for stability (i.e., positive recurrence) of the model. 
This proves the conjecture on the stability region in~\cite{storm2020}. Our 
main result also implies that the stability condition is sufficient for the 
slotted-ring model studied in~\cite{vArem1990} to be stable in a particular 
sense (see Section~\ref{subsec: slotted-ring model} for details). 
In~\cite{vArem1990}, this result was only shown to hold under the rather 
strong additional assumption that, among other things, sequences of certain 
characteristic inter-event times of the system are asymptotically strongly 
stationary, and the ergodic rates at which these events occur are 
well-defined. Our paper proves that this assumption is unnecessary, and is in 
fact implied by the stability condition.

\bigskip

Stability (or positive recurrence) of a Markov process is a fundamental 
condition for many theorems in Markov chain theory that are commonly used to 
study stationary or asymptotic properties. The mathematical theory of 
stability is well-developed~\cite{bremaud2013, MT2012, revuz2008}. However, 
for queueing models that involve multiple customer classes and priorities, 
like the one in this paper, determining conditions for stability remains a 
notoriously difficult task. In particular, many instances with such features 
exist where a subcritical system load (i.e., an arrival rate strictly below~1 
of each station's normalized workload) is not sufficient for 
stability~\cite{bramson1994, kumar1989, lu1991, rybko, seidman1994}. Moreover, 
the model that we consider does not have a product-form stationary 
distribution~\cite{storm2020}, which makes it impossible to determine a 
stability condition from the normalization of the stationary measure. As it 
turns out, though, we can leverage the specific routing topology of a ring 
model to formally establish the stability condition.

An important consequence of the necessity and sufficiency of the stability 
condition is that this condition can be interpreted as the \textit{system 
capacity}. Specifically, it implies that the set of arrival rates for which 
the model is stable is \textit{monotone}, meaning that the system does not 
lose its property of positive recurrence when arrival rates are decreased. 
This monotonicity is not automatic for multiclass networks with priorities, 
see~\cite{dai1999} for an example. In the context of communication systems and 
road traffic, capacity is an important measure for the load that the system is 
able to process; see, e.g.,~\cite{bramson2, li} for similar recent work in 
this area.

We prove the stability condition for the model in~\cite{storm2020} using a 
fluid model approach. More precisely, we establish a coupling between the 
Markov chain associated with the model and the queue length process of a 
multiclass queueing network, such that they follow the same sample paths up to 
a bijection. As a consequence, proving positive recurrence of the Markov chain 
can be reduced to proving stability of the simpler fluid model associated with 
the multiclass network~\cite{dai, bramson}. We prove the stability of the 
fluid model with an approach based on~\cite{tassiulas1996}, with some 
modifications that are required in our context. This approach exploits the 
relationship between the stability condition and the marginal stationary rate 
at which segments of the ring are occupied when the model is stable.

\bigskip

As mentioned earlier, stochastic ring networks that involve queues are 
ubiquitous within communication and transportation systems. In these domains, 
priority service for traffic on the ring over exogenously arriving traffic is 
natural to create free flow on the ring, so as to avoid collisions and 
eliminate the need for buffering within the ring. Additionally, the 
destination of a packet (or vehicle) typically depends on where it entered the 
ring, giving rise to \textit{multiple customer types}. The combination of 
these properties (a ring, a priority structure, and multiple customer types) 
motivates the specific model we consider in this paper. 

In the context of communication networks, ring topologies are the cornerstone 
of various widely deployed architectures, and our model pertains to LANs with 
slotted-ring MAC protocols~\cite{coffman1998, vArem1990, vanArem1990b, 
zafirovic1988} as well as Wavelength-Division Multiplexing (WDM) based 
Metropolitan Area Networks~(MANs) with packet switching or optical burst 
switching, see \cite[Section~4.1]{bogaerts2012} and~\cite{Fiems2016, 
herzog2004, yang2004, yuang2010}. Nowadays, ring-topology waveguides are a 
standard component in optical networks because of their ability to transmit 
large amounts of data via light in a short amount of time~\cite{chremmos2010, 
Jara2020, Singh2018}. While current networks mostly rely on wavelength 
routing, optical burst switching and packet switching mechanisms offer finer 
granularity and more dynamic resource sharing, and thus allow for higher 
efficiency and bandwidth utilization, especially with bursty and unpredictable 
traffic patterns.

Ring structures are also encountered in traffic networks as a form of 
intersection design, known as a roundabout. On a roundabout, vehicles traverse 
an intersection via a circulating ring to which several roads are attached 
that function as on- and off-ramps. Vehicles on the circulating ring have 
priority over vehicles in the attached legs, cf.~\cite{belz2016, storm2020}. 
These properties make them very similar, both in function and dynamics, to 
rings in communication systems, as described above. The results in this paper 
apply to a model that was originally proposed for single-lane 
roundabouts~\cite{storm2020}.

\bigskip

The outline of the paper is as follows. We start by introducing the model in 
Section~\ref{sec:model}. Section~\ref{sec: preliminaries} provides a 
preliminary analysis, introduces the stability condition, and states our main 
theorem. In Section~\ref{sec: multiclass network description} we couple the 
model to a multiclass queueing network. We then proceed by proving stability 
of this multiclass queueing network in Section~\ref{sec: proof main result} 
using fluid limits. Finally, Section~\ref{sec: related models} discusses 
related models to which our results can be applied directly and possible 
extensions to other models, and Section~\ref{sec: conclusions} concludes.

\emph{Notation.} Unless otherwise specified, all random variables and 
processes are defined on a common probability space~$(\,\Omega,\cF,\PP\,)$. We 
write $\Z_+$ for the non-negative integers, $\N$ for the natural numbers 
(i.e., $\N \equiv \Z_+ \setminus \{0\}$), $\R$ for the real numbers, and 
$\R_+$ for the set of non-negative reals. For any dimension~$d$, 
$\norm{\,\cdot\,}$ denotes the $L^1$-norm on elements of~$\R^d$. If $a$ 
and~$b$ are real numbers, then $a \bmin b$ is the minimum and $a\bmax b$ the 
maximum of $a$ and~$b$.

\section{Model description}
\label{sec:model}

As we mentioned in the introduction, the model we consider is a stochastic 
ring network that has a variety of application domains, but was originally 
introduced as a model for a roundabout in~\cite{storm2020}. In this section, 
to describe the model, we stay close to the original formulation as a 
roundabout model and use the related (road traffic) terminology. For technical 
reasons, which we will explain in Section~\ref{sec: related models}, we 
consider a version of the model that differs slightly from the roundabout 
model in~\cite{storm2020}. For the model we consider, we identify and prove a 
necessary and sufficient condition for stability. This result also identifies 
the global stability regions for the model in~\cite{storm2020} and the 
slotted-ring model in~\cite{vArem1990}, as we formally show in 
Section~\ref{sec: related models}.

The roundabout model is a slotted ring consisting of $L$~cells, with on-ramp 
queues in front of each cell. The cells and queues are indexed by $i \in 
\{1,\ldots,L\}$, with cell~1 adjacent to cell~$L$; using this cyclic 
structure, we allow ourselves to use the index~$i +\nobreak L$ to refer to 
cell/queue~$i$. The presence of vehicles on the roundabout is modeled by the 
state of the cells. Each cell can either be empty, or contain a vehicle that 
has entered the roundabout at some cell~$j$. In the first case we say the 
state of cell~$i$ is~0 and in the second case we say that its state is~$j$; we 
will also say that a cell is occupied by a vehicle of \emph{type}~$j$ when the 
state of the cell is~$j$. The queues model vehicles waiting to enter the 
roundabout, and their state is a number in~$\Z_+$.

The model has discrete-time dynamics. The main idea is that vehicles arrive 
via the queues to the roundabout, traverse a number of cells to an off-ramp, 
and depart the roundabout. At each time~$t \in\nobreak \Z_+$, if there are 
vehicles in queue~$i$ and cell~$i$ is empty, one vehicle from the queue will 
enter the roundabout and occupy cell~$i +\nobreak 1$ at time~$t +\nobreak 1$. 
If cell~$i$ is occupied by a vehicle of type~$j$, then the vehicle will either 
depart from the system with probability~$q_{ij} \in [0,1]$, or occupy cell~$i 
+\nobreak 1$ at the next time step. We emphasize that the probability~$q_{ij}$ 
depends on~$j$ to reflect that the cell at which a vehicle leaves the system 
may depend on where it entered. To model arriving vehicles, we impose that at 
every time step, a new vehicle will arrive to queue~$i$ with probability~$p_i 
\in [0,1]$. We assume that all decisions whether a vehicle arrives at a 
cell~$i$ and whether a vehicle of type~$j$ will leave cell~$i$ (if present), 
are made independently of each other and of the current state and history of 
the process (see Appendix~\ref{sec: appendix equality markov chains} for an 
explicit construction of the process). Note that on-ramps and off-ramps can be 
removed by setting arrival or departure probabilities equal to zero.
 
We now provide a more precise account of these dynamics by formulating update 
rules for the system (as is customary for cellular automata). These rules are 
\textit{local} in the sense that we only need to know the current joint states 
of queue~$i$ and cell~$i$, and can disregard the remainder of the system, to 
determine the new states of queue~$i$ and cell~$i +\nobreak 1$ (in accordance 
with the cellular automata paradigm). It turns out that we need to distinguish 
three cases:
\begin{list}
	{\textbf{Case~\arabic{enumi}:}}
	{\usecounter{enumi}\setlength{\leftmargin}{0pt}%
	 \setlength{\labelwidth}{-\labelsep}}
\item cell~$i$ and queue~$i$ are both empty. In this case, no vehicle can 
	enter the roundabout from queue~$i$ and cell~$i +\nobreak 1$ will be empty 
	the next time step. If a new vehicle arrives to queue~$i$ (which happens 
	with probability~$p_i$), then queue~$i$ will have length~1 at the next 
	time step. Otherwise, it will have length zero.
\item cell~$i$ is empty and queue~$i$ is not empty. Then, the vehicle at the 
	front of queue~$i$ will enter the roundabout and move on to cell~$i 
	+\nobreak 1$ in the process. Thus, cell~$i +\nobreak 1$ will be in 
	state~$i$ one unit of time later. Queue~$i$ either remains at its current 
	length (if a new vehicle arrives to queue~$i$), or its length decreases by 
	one.
\item cell~$i$ is occupied by a type-$j$ vehicle. In this case, queue~$i$ is 
	blocked, hence its length will grow by one if a new vehicle arrives, or 
	will stay the same otherwise. The vehicle in cell~$i$ will either leave 
	the system with probability~$q_{ij}$, in which case cell~$i +\nobreak 1$ 
	will be empty one unit of time later, or the vehicle remains in the 
	system, so that cell~$i +\nobreak 1$ will be in state~$j$ at the next time 
	step.
\end{list}

Having specified the precise dynamics of the model, we observe that if for 
type-$j$ vehicles, $q_{ij} = 0$ for each~$i \in \{1,\ldots,L\}$, then such 
vehicles cannot leave the network. We therefore require that $\prod_{\ell=1}^L 
(1-q_{\ell j}) < 1$ for each~$j \in \{1,\ldots,L\}$, so that any vehicle 
eventually leaves the network. Additionally, note that when $p_i = 1$ for some 
$i \in \{1,\ldots,L\}$, then queue~$i$ can never decrease in length. Since our 
focus is on stability, we impose $p_i \in [0,1)$ for all~$i \in 
\{1,\ldots,L\}$.

\section{Stability condition and main result}
\label{sec: preliminaries}

The model under consideration is a discrete-time Markov chain, which we denote 
by $X^1(\cdot) = \{X^1(t) \colon t \in \Z_+ \}$. The associated state space is 
$\cS^1 := (\Z_+ \times \{0,\ldots,L\})^L$, the set of vectors in~$\R^{2L}$ 
that describe the state of each cell and queue. We observe that $X^1(\cdot)$ 
is irreducible on the set of states that can be reached from the empty state, 
since from any state, with strictly positive probability, we can empty the 
system in a finite number of steps. Note that specific choices of the $p_i$ 
and~$q_{ij}$ can make it impossible to reach all states in the state space. In 
addition, $X^1(\cdot)$ is aperiodic as we can remain in the empty state for an 
arbitrary (finite) number of time steps, with strictly positive probability.

Following \cite{bramson, dai, MT2012}, we say that $X^1(\cdot)$ is 
\textit{stable} when it is positive recurrent. In Section~\ref{sec: marg stat 
dist}, we will derive an explicit expression for the marginal stationary 
distribution of the states of the cells under the assumption of stability. 
This result was already obtained in~\cite{storm2020}, but here we use a 
different method based on the system's offered load. To be more precise, we 
show that when the system is stable, the marginal stationary probability that 
cell~$i$ is in state~$j$ is equal to $p_j$ times the expected number of times 
that a vehicle of type~$j$ will occupy cell~$i$ during the time it spends on 
the roundabout. This connection leads us to formulate a necessary and 
sufficient condition for stability in Section~\ref{sec: stability condition}, 
which is our main result. Moreover, this connection plays an important role in 
our proof that the condition is sufficient in Section~\ref{sec: proof main 
result}. We prove the necessity of the condition for stability at the end of 
Section~\ref{sec: stability condition}.

\subsection{Marginal stationary distribution}
\label{sec: marg stat dist}

In our formulation of the model given above, we update the system at each time 
step by checking for external arrivals, and deciding for each vehicle on the 
roundabout whether or not it leaves. As a consequence of the dynamics, the 
time spent on the roundabout by each vehicle that enters the roundabout at 
cell~$j$ is an independent copy of a generic random variable~$T_j$, the 
distribution of which is completely determined by the parameters~$q_{ij}$. We 
use this fact in this subsection to derive the marginal stationary 
distribution of the states of the cells when the system is stable, and again 
in Section~\ref{sec: proof main result} to prove our main result.

To be precise, $T_j$ is a random variable with distribution given by
\begin{equation}\label{Eqn: T_j}
\begin{cases}
	\PP(T_j = 0) = 0; \\[5pt]
	\PP(T_j = k) = \prod_{\ell = j+1}^{j+k-1} (1-q_{\ell j})\, q_{j+k,j}
	&\text{for $k\geq 1$},
\end{cases}
\end{equation}
where we adopt the convention that the empty product is equal to~1, and define 
$q_{\ell j}$ for $\ell>L$ by setting $q_{\ell j} := q_{ij}$ whenever $\ell 
\equiv i \pmod{L}$. Observe that the model dynamics determine the distribution 
of the~$T_j$ and that different distributions could be considered as well. 
Such model extensions do not significantly impact our analysis and are 
discussed in Section~\ref{subsec: model extensions}.

Observe that the number of times a vehicle of type~$j$ that spends time~$T_j$ 
on the roundabout will occupy cell~$i$ before leaving, is also a well-defined 
random variable. We denote it by~$N_{ij}$, and we set $b_{ij} := \E N_{ij}$. 
In our model, $b_{ij}$ is the expected number of visits to cell~$i$ by a 
vehicle that moves onto the roundabout from the on-ramp at cell~$j$. We next 
define, for $i\in \{1,2, \dots, L\}$,
\begin{equation}
	\label{Eqn: pi_ij}
	\pi_{ij} := \begin{cases}
		b_{ij} \, p_j & \text{if $j \in \{1,2,\dots,L\}$}; \\
		1 - \sum_{j=1}^L b_{ij} \, p_j & \text{if $j=0$}.
	\end{cases}
\end{equation}
The numbers~$\pi_{ij}$ have an important interpretation in case the model is 
stable:

\begin{proposition}[Marginal stationary distribution]
	\label{Prop: marginal distribution}
	If the model is stable, then the marginal stationary probability that 
	cell~$i$ is in state~$j$ is the number~$\pi_{ij}$ defined by~\eqref{Eqn: 
	pi_ij}.
\end{proposition}

\begin{proof}
	Suppose that the model is stable. Then the long-term rate at which 
	type-$j$ vehicles move onto the roundabout must be equal to the rate at 
	which such vehicles arrive, which is~$p_j$ by the strong law of large 
	numbers. Again by the strong law of large numbers, it follows that the 
	long-term fraction of time cell~$i$ is occupied by a type-$j$ vehicle 
	is~$p_j\, b_{ij}$. By the ergodic theorem, this number is also the 
	marginal stationary probability that cell~$i$ is in state~$j$.
\end{proof}

As we show below, for our model the numbers~$b_{ij}$ can be expressed 
explicitly in terms of the parameters~$q_{ij}$ by the following formula:
\begin{equation}
	\label{Eqn: b_ij}
	b_{ij}
	= \begin{cases}
		\displaystyle
		\frac{\prod_{\ell=j+1}^{i+L-1} (1-q_{\ell j})}
			{1 - \prod_{\ell=1}^L (1-q_{\ell j})}
		& \quad\text{if $1\leq i\leq j\leq L$}; \\[18pt]
		\displaystyle
		\frac{\prod_{\ell=j+1}^{i-1} (1-q_{\ell j})}
			{1 - \prod_{\ell=1}^L (1-q_{\ell j})}
		& \quad\text{if $1\leq j < i\leq L$}.
	\end{cases}
\end{equation}
For example, if $L=1$ we have $b_{11} = q_{11}^{-1}$ because the time a 
vehicle spends on the roundabout follows a geometric distribution with 
parameter~$q_{11}$. For $L=2$, formula~\eqref{Eqn: b_ij} gives
\[
	b_{12} = \frac{1}{1-(1-q_{12})(1-q_{22})}
	\quad \text{and} \quad
	b_{22} = \frac{1-q_{12}}{1-(1-q_{12})(1-q_{22})}
\]
for vehicles of type~2. This is explained as follows. Note that a vehicle of 
type~2 first enters cell~1 when it moves onto the roundabout. The vehicle's 
return probability to cell~1 is $(1-q_{12})(1-q_{22})$, hence the above 
expression for~$b_{12}$ is indeed the expected number of visits to cell~1. The 
expected number of visits to cell~2 is then $(1-q_{12}) \, b_{12}$ because 
$1-q_{12}$ is the probability the vehicle moves on to cell~2 when it occupies 
cell~1. We now prove~\eqref{Eqn: b_ij} in the general case:

\begin{lemma}
	\label{Lemma: EN_ij}
	For all $i,j\in \{1,\dots,L\}$ we have that $b_{ij} = \E N_{ij}$ is given 
	by~\eqref{Eqn: b_ij} if $T_j$ has the distribution specified 
	by~\eqref{Eqn: T_j}.
\end{lemma}

\begin{proof}
	Consider the case $1\leq j < i\leq L$. Note that the probability 
	distribution of~$N_{ij}$ can be expressed in terms of the distribution 
	of~$T_j$ as
	\[
		\PP(N_{ij}\geq n+1) = \PP(T_j \geq i-j+nL),
		\qquad n\geq0.
	\]
	From~\eqref{Eqn: T_j} it follows by induction in~$k$ that
	\[
		\PP(T_j \geq k) = \prod_{\ell=j+1}^{j+k-1} (1-q_{\ell j}),
		\qquad k\geq1.
	\]
	Combining these observations, we obtain
	\[
		\PP(N_{ij}\geq n+1)
		= \prod_{\ell=j+1}^{i-1}(1-q_{\ell j})
			\biggl(\, \prod_{\ell=1}^L (1-q_{\ell j}) \biggr)^n,
	\]
	and summing both sides over all~$n\geq0$ yields formula~\eqref{Eqn: b_ij} 
	for~$b_{ij} = \E N_{ij}$. The proof in the case $1\leq i\leq j\leq L$ is 
	similar.
\end{proof}

\begin{remark}\label{rem: EN_ij}
	We saw in the proof of Proposition~\ref{Prop: marginal distribution} 
	that~$\pi_{ij}$ represents the long-term fraction of time cell~$i$ is 
	occupied by a type-$j$ vehicle when the model is stable. But the 
	quantity~$\pi_{ij}$ is also meaningful if the system is unstable, since it 
	is by definition only a measure of the load imposed on the system. The 
	question of stability deals with whether or not all the vehicles manage to 
	get onto the roundabout. If the system is unstable, the~$\pi_{ij}$ still 
	describe the load that is offered to the system, but we cannot use 
	ergodicity to conclude that~$\pi_{ij}$ is the fraction of time cell~$i$ 
	will be occupied by a vehicle of type~$j$. Our characterization of 
	the~$\pi_{ij}$ in terms of~$\E N_{ij}$ plays an important role in proving 
	the main theorem (Theorem~\ref{Thm: (main) stability of the model} below).
\end{remark}

\subsection{Condition for stability}
\label{sec: stability condition}

Under the assumption that the system is stable, the Markov chain~$X^1(\cdot)$ 
has a stationary distribution~$\pi$. In the previous section we have shown 
that the vector $\{\pi_{i0}, \pi_{i1}, \dots, \pi_{iL}\}$ given by \eqref{Eqn: 
pi_ij} and~\eqref{Eqn: b_ij} is then the marginal stationary distribution of 
cell~$i$, for $i = 1,2,\dots,L$. In particular, $\pi_{i0}$ is the stationary 
rate at which cell~$i$ is empty when the system is stable. Since vehicles 
arrive at cell~$i$ at rate~$p_i$ and can only enter onto the roundabout when 
the cell is empty, it seems reasonable to believe that the system cannot be 
stable if $p_i > \pi_{i0}$ for some cell~$i$. Conversely, one would suspect 
that if $p_i < \pi_{i0}$ for all cells~$i$, the cells will be vacant often 
enough to prevent the queues from blowing up. It turns out that the latter 
condition is in fact not only sufficient, but also necessary for the system to 
be stable:

\begin{theorem}[Main result]
	\label{Thm: (main) stability of the model}
	A necessary and sufficient condition for the roundabout model to be stable 
	(by which we mean that the Markov chain~$X^1(\cdot)$ is positive 
	recurrent) is that
	\[
		p_i < \pi_{i0} = 1 - \sum_{j=1}^L b_{ij} \, p_j
		\qquad \text{for all~$i\in \{1,2,\dots,L\}$}.
	\]
\end{theorem}

For example, in the case $L=1$, the expected time a vehicle occupies the cell 
is $b_{11} = q_{11}^{-1}$, as we saw in Section~\ref{sec: marg stat dist}. It 
follows that the expected time that elapses after a vehicle enters the 
roundabout until the next vehicle can enter is $1+q_{11}^{-1}$, because the 
cell has to be empty for at least one unit of time before it can be occupied 
again. Hence, for the system to be stable, the arrival rate~$p_1$ must be 
smaller than $1 / \bigl( 1+q_{11}^{-1} \bigr)$, which is precisely what the 
stability condition of Theorem~\ref{Thm: (main) stability of the model} says. 
Now consider the case~$L=2$. Suppose we fix the arrival rate~$p_2$. Then 
vehicles of type~2 already lay a claim on cell~1 for a fraction~$b_{12} \, 
p_2$ of the time. During the remaining fraction of time, like in the case 
$L=1$, the maximum rate at which cell~1 can be empty if we also allow type-1 
vehicles onto the roundabout is $(1+b_{11})^{-1}$. This leads to the stability 
condition $p_1 < (1-b_{12} \, p_2) (1+b_{11})^{-1}$, and an analogous 
condition for~$p_2$ can be derived by reversing the roles of the type-1 and 
type-2 vehicles. The stability region consists of the pairs~$(p_1, p_2)$ that 
satisfy both inequalities (see Figure~\ref{fig: stable region}). In general, 
for any~$L$, the stability condition is a system of linear inequalities for 
the arrival rates~$p_i$, and hence the stability region is a convex polytope 
contained in~$[0,1)^L$ with the zero vector as one of its corners.

\begin{figure}
	\begin{center}
		\includegraphics{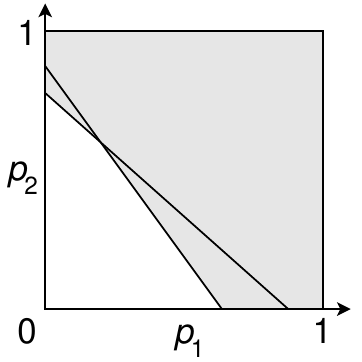}
	\end{center}
	\caption{The stability region (unshaded area) for $L=2$, $q_{11} = q_{12} 
	= 0.75$, $q_{21} = q_{22} = 0.5$.}
	\label{fig: stable region}
\end{figure}

The hard part of the theorem is to prove that the stated condition is indeed 
sufficient for the stability of the system. We close this section with a proof 
of the necessity of the stability condition; Sections \ref{sec: multiclass 
network description} and~\ref{sec: proof main result} are devoted to proving 
the sufficiency. One may also wonder about transience in case the stability 
condition does not hold; this is discussed in Section~\ref{sec: conclusions}.

\begin{proposition}\label{Prop: necessity main theorem}
	The condition that $p_i < \pi_{i0}$ for all~$i\in \{1,2,\dots,L\}$ is a 
	necessary condition for the stability of the roundabout model.
\end{proposition}

\begin{proof}
	We know that the Markov chain~$X^1(\cdot)$ is aperiodic and irreducible. 
	Assume furthermore that the system is stable, that is, that $X^1(\cdot)$ 
	is also positive recurrent. Then $X^1(\cdot)$ has a stationary 
	distribution~$\pi$, and since the Markov chain is irreducible, the state 
	in which every cell and every queue is empty must have a strictly positive 
	probability, which we denote by~$\pi_{\varnothing}$, under the stationary 
	distribution.

	Suppose that we start our Markov chain from the stationary 
	distribution~$\pi$. Let $C_i(t)$ denote the state of cell~$i$ at time~$t$, 
	and let $Q_i(t)$ be the length of queue~$i$ at time~$t$. Write $A_i(t)$ 
	for the event that a new vehicle arrives at cell~$i$ at time~$t$. Then the 
	queue length process satisfies
	\[\begin{split}
		Q_i(t+1)
		&= Q_i(t) + \1_{\{Q_i(t)>0\}} \bigl( \1_{A_i(t+1)} - \1_{\{C_i(t)=0\}} 
		\bigr)  + \1_{\{Q_i(t)=0\}} \1_{A_i(t+1)} \\
		&= Q_i(t) + \1_{A_i(t+1)} - \1_{\{C_i(t)=0\}} + \1_{\{Q_i(t) = C_i(t) 
		= 0\}},
	\end{split}\]
	from which it follows that
	\begin{equation}\label{Eqn: Q_i(n)/n}
		\frac1n Q_i(n)
		= \frac1n Q_i(0) + \frac1n \sum_{t=0}^{n-1} \1_{A_i(t+1)}
		- \frac1n \sum_{t=0}^{n-1} \1_{\{C_i(t)=0\}} + \frac1n 
		\sum_{t=0}^{n-1} \1_{\{Q_i(t) = C_i(t) = 0\}}.
	\end{equation}

	We now want to take $n\to\infty$ in~\eqref{Eqn: Q_i(n)/n}. On the one 
	hand, since we assumed the model is stable, it is clear that we must have 
	that
	\[
		\liminf_{n\to\infty} \frac1n Q_i(n) = 0 \qquad \text{a.s.}
	\]
	On the other hand, the first term on the right hand side in~\eqref{Eqn: 
	Q_i(n)/n} clearly converges almost surely to~0 as $n\to\infty$, whereas 
	the second term converges a.s.\ to~$p_i$ by the strong law of large 
	numbers, and the third term converges a.s.\ to~$-\pi_{i0}$ by the ergodic 
	theorem. As for the last term in~\eqref{Eqn: Q_i(n)/n}, we observe that 
	the event that the system is completely empty is a subset of the event 
	$\{Q_i(t) = C_i(t) = 0\}$, so that by the ergodic theorem we can conclude 
	that
	\[
		\liminf_{n\to\infty} \frac1n \sum_{t=0}^{n-1} \1_{\{Q_i(t)=C_i(t)=0\}} 
		\geq \pi_{\varnothing} > 0.
	\]
	Combining these observations, it follows that almost surely,
	\[
		0
		= \liminf_{n\to\infty} \frac{1}{n} Q_i(n)
		\geq 0 + p_i - \pi_{i0} + \pi_{\varnothing}
		> p_i - \pi_{i0}.
	\]
	Hence, $p_i$ must be strictly smaller than~$\pi_{i0}$ for every cell~$i$ 
	if the roundabout model is stable.
\end{proof}

\section{Multiclass queueing network formulation}
\label{sec: multiclass network description}

To prove sufficiency of the stability condition in Theorem~\ref{Thm: (main) 
stability of the model}, we formulate in this section a multiclass network 
(explained below) that has essentially the same dynamics as the roundabout 
model from Section~\ref{sec:model}. That is, for each choice of parameters 
$p_i$ and~$q_{ij}$, $i,j \in \{1,\dots,L\}$, for the roundabout model we 
define an analogous multiclass network, using the same parameters to define 
the appropriate exogenous arrival and routing processes in the network. We 
shall refer to a choice of these parameters as a \textit{parameter setting}. 
The two model formulations can be coupled on a sample-path level, formalized 
in Lemma~\ref{Lemma: coupling} below. This allows us in Section~\ref{sec: 
proof main result} to utilize the powerful fluid model framework for 
multiclass queueing networks to prove stability of the multiclass-network 
formulation of our model, hence proving it for the model in the original 
formulation as well.

A \textit{multiclass queueing network}, or simply \textit{multiclass network}, 
is a network consisting of a finite collection of (single server) stations 
serving customers from a finite number of customer classes. Each customer 
class has its own queue at one of the stations, with its own exogenous arrival 
process and service time distribution. After service completion, customers are 
either routed to another queue (and associated customer class), or depart from 
the network. In addition, a multiclass network can have other characteristics 
relating to specific model applications, such as customer class priorities in 
stations and blocking features.

Fix a parameter setting. To map the roundabout model to a multiclass 
network, we identify each cell~$i$ with a station~$i$ in the multiclass 
network. Vehicles that arrive from outside to the on-ramp of cell~$i$ become 
customers of class~$(i,0)$ in the multiclass network, and vehicles of type~$j$ 
that occupy cell~$i$ become class~$(i,j)$ customers. Thus, there are 
$L$~stations and $L(L +\nobreak 1) = L^2 +\nobreak L$ customer classes in the 
multiclass network. The set of customer classes is
\[
	\cK := \bigl\{ (i,j) \colon i \in \{1,\ldots,L\}, j\in\{0,\ldots,L\} 
	\bigr\},
\]
and station~$i$ serves the customers of all classes~$(i,j)$ with $j \in 
\{0,\dots,L\}$. We will now define the exogenous arrival processes, service 
time distributions, routing policy, and priority structure that will make the 
multiclass network mimic the dynamics of the roundabout model. As an aid to 
the reader, Figure~\ref{fig: Network} illustrates the multiclass network and 
its routing policy.

\begin{figure}
	\centering
	\includegraphics{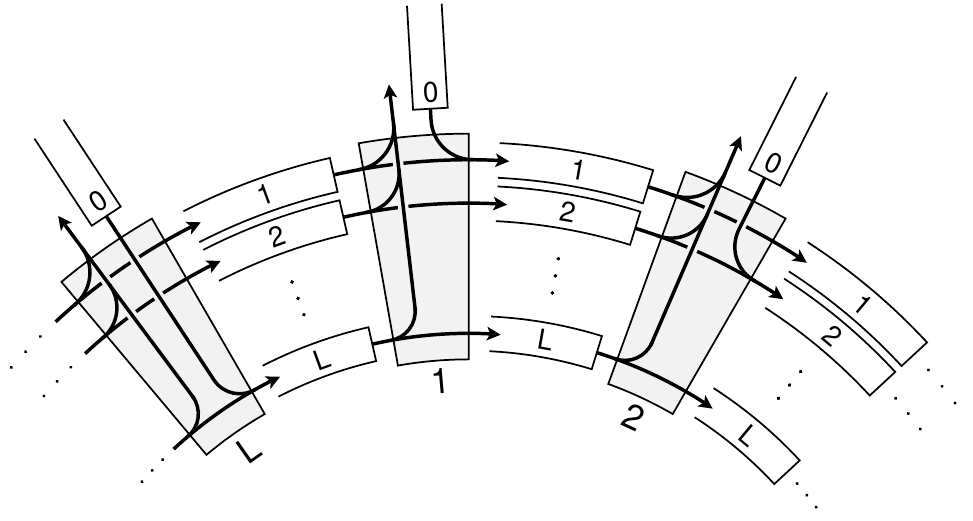}
	\caption{The multiclass queueing network. The shaded boxes are the 
	stations and the queues are labeled with the corresponding vehicle type 
	(using~$0$ for external arrivals). The arrows show how customers can be 
	routed through the network from the $L+1$ queues at each station.}
	\label{fig: Network}
\end{figure}

First, to account for the external arrivals, the exogenous arrival process for 
class~$(i,0)$ customers clearly must have independent, geometrically 
distributed inter-arrival times with parameter~$p_i$. All other customer 
classes~$(i,j)$ with $i,j \in \{1,\dots,L\}$ have no exogenous arrivals; such 
customers can appear in the network only because they arrive as a 
class~$(j,0)$ customer and are subsequently routed in the network to become a 
class~$(i,j)$ customer. This corresponds in the roundabout model to a vehicle 
entering onto the roundabout at cell~$j$, and then traversing a number of 
cells to reach cell~$i$.

To mimic this behavior correctly in the multiclass network, we impose that the 
service times are exactly one unit of time for each class of customer. 
Furthermore, after completing service at station~$i$, a class~$(i,0)$ customer 
is routed to the next station and becomes a customer of class $( (i \bmod 
L)+\nobreak 1, i)$. Likewise, a class~$(i,j)$ customer is either routed to the 
next station with probability~$1-\nobreak q_{ij}$, or leaves the network. This 
routing policy is captured by the routing matrix~$P$. The rows and columns of 
this matrix are indexed by the set of customer classes~$\cK$, and the 
entry~$P_{ij, k\ell}$ defines the probability that a class~$(i,j)$ customer is 
routed to become a class~$(k,\ell)$ customer after service completion, and is 
therefore given by
\begin{equation}\label{Eqn: routing matrix}
	P_{ij,k\ell} = \begin{cases}
		1			& \text{if $j=0$, $k = (i \bmod L)+1$, $\ell=i$;} \\
		1-q_{ij}	& \text{if $j\neq0$, $k = (i \bmod L)+1$, $\ell=j$;} \\
		0			& \text{otherwise.}
	\end{cases}
\end{equation}

Next, we need to account for the fact that in the roundabout model, any 
vehicle that occupies a cell blocks the corresponding on-ramp and is 
guaranteed to move out of that cell after one unit of time. We therefore 
impose the following priority structure on the multiclass network: for each~$i 
\in \{1,\dots,L\}$, customers of classes~$(i,j)$ with $j \in \{1,\dots,L\}$ 
have priority over class~$(i,0)$ customers. This means that class~$(i,0)$ 
customers only receive service at station~$i$ if the queues associated with 
class~$(i,j)$ customers with $j\neq 0$ are all empty.

Finally, we need to consider the restriction in the roundabout model that each 
cell can be occupied by at most one vehicle at a time. For the multiclass 
network this means that we have to guarantee that, for each station~$i$, the 
combined number of customers of classes~$(i,j)$ with $j\neq 0$ is at most~1 at 
all times. But we can achieve this by simply imposing this condition on the 
initial state of the network: if at time~0, for each station~$i$, the combined 
number of customers of classes~$(i,j)$ with $j \in \{1,\dots,L\}$ is at 
most~1, then the same is automatically true for all times~$t \in\nobreak 
\Z_+$, because all service times are exactly~1, class~$(i,j)$ customers with 
$j\neq0$ have priority over class~$(i,0)$ customers, and customers can only be 
routed to one of the queues at station~$i$ by the previous station in the 
network. Henceforth, we only consider the multiclass network under this 
condition on the initial state.

It is well known (see, e.g., \cite[Section~2] {dai} or \cite[Section~4.1] 
{bramson}) that a multiclass network can be represented as a Markov chain. In 
our case, due to the geometric inter-arrival times, deterministic service 
times and discrete-time nature of the network, we can represent it as a 
discrete-time Markov chain with a discrete state space. We denote this Markov 
chain by~$X^2(\cdot)$, and take as our state space $\cS^2 := (\Z_+ \times 
\mathcal{V})^L$, where $\mathcal{V} := \{ x \in \{0,1\}^L \colon \sum_{k} x_k 
\leq 1 \}$. Thus, each state is a vector in~$\R^{ \scriptscriptstyle L^2 + L}$ 
that tells us, for every customer class, how many of those customers are 
present in the network. Because of the way we constructed the multiclass 
network, there is a natural bijection between the two state spaces $\cS^1$ 
and~$\cS^2$, and it is possible to couple the two Markov chains $X^1(\cdot)$ 
and~$X^2(\cdot)$ in such a way that they follow the same sample paths up to 
this bijection. We formalize this claim in the following lemma, the proof of 
which is given for completeness in Appendix~\ref{sec: appendix equality markov 
chains}:

\begin{lemma}
	\label{Lemma: coupling}
	For any fixed parameter setting, there exists a coupling of $X^1(\cdot)$ 
	and~$X^2(\cdot)$ and a bijection~$f \colon \cS^1 \to \cS^2$ such that for 
	every~$\omega \in \Omega$, if $f\bigl( X^1(\omega,0) \bigr) = 
	X^2(\omega,0)$, then
	\[
		f\bigl( X^1(\omega,t) \bigr) = X^2(\omega,t)
		\qquad\text{for all~$t \in \Z_+$.}
	\]
\end{lemma}

Lemma~\ref{Lemma: coupling} allows us to complete the proof of 
Theorem~\ref{Thm: (main) stability of the model} by proving that the stability 
condition implies stability of~$X^2(\cdot)$. For this we can rely on the 
theory of fluid limits and Foster--Lyapunov functions, cf.~\cite{bramson, 
bramson2, dai, li}.

\section{Proof of the main result}
\label{sec: proof main result}

This section of the paper is dedicated to proving that $X^2(\cdot)$ is stable 
if the stability condition in Theorem~\ref{Thm: (main) stability of the model} 
is satisfied. By Lemma~\ref{Lemma: coupling} and Proposition~\ref{Prop: 
necessity main theorem}, this completes the proof of Theorem~\ref{Thm: (main) 
stability of the model}. Nowadays there is a standardized approach for proving 
stability of multiclass queueing networks, which utilizes the powerful 
framework of \emph{fluid limits} and stability of the associated \emph{fluid 
model}. For a complete exposition of this technique we refer 
to~\cite{bramson}, which is based on work that goes back to~\cite{rybko, dai, 
bramson1996}. We begin our proof in Section~\ref{sec: queueing equations} by 
writing down the set of \emph{queueing equations} that describe the dynamics 
of the multiclass network in terms of simple stochastic processes, for which 
we derive the associated fluid model in Section~\ref{sec: fluid model}. We 
then complete the proof of stability in Section~\ref{sec: stability} by making 
use of the fluid model.

\subsection{The queueing equations}
\label{sec: queueing equations}

For $i$ in~$\{1,\ldots,L\}$, let $A_{i0}(\cdot) := \{A_{i0}(t)\colon t \in 
\Z_+\}$ be the stochastic process that counts the cumulative number of 
exogenous class~$(i,0)$ arrivals up to time~$t$. For $i,j$ in~$\{1,\ldots,L\}$ 
we set $A_{ij}(\cdot) \equiv 0$, reflecting that the customer class~$(i,j)$ 
has no exogenous arrivals. Let $Q^x_{ij}(t)$ denote the length of the queue 
containing class~$(i,j)$ customers at time~$t \in \Z_+$, where $x$ denotes the 
initial state of the system. This defines the queue length processes 
$Q^x_{ij}(\cdot) := \{Q^x_{ij}(t)\colon t \in \Z_+\}$. Note that they can be 
combined into a vector-valued process that has the same law as the Markov 
chain~$X^2 (\cdot)$ starting from the state~$x$. We set $T^x_{ij}(\cdot) := 
\{T^x_{ij}(t)\colon t \in \Z_+\}$, where $T^x_{ij}(t)$ is the cumulative 
amount of service time that has been spent by station~$i$ on class~$(i,j)$ 
customers up to time~$t$. As the service times of each customer class are 
precisely one unit of time, $T^x_{ij}(t)$ is also equal to the total number of 
class~$(i,j)$ service completions up to time~$t$. 

Recall that $\cK$ is the set of customer classes. For each class~$(i,j)$ we 
introduce the random vector~$\phi^{ij}(n)$ with components $\phi^{ij}_{k\ell} 
(n)$, $(k,\ell) \in \cK$, where $\phi^{ij}_{k\ell}(n) = 1$ if the $n$th 
class~$(i,j)$ customer served by station~$i$ is routed to the class~$(k,\ell)$ 
queue, and $\phi^{ij}_{k\ell}(n) = 0$ otherwise. Observe that $\phi^{ij}(n)$ 
is an $(L^2 +\nobreak L)$-dimensional random vector with expectation 
$(P_{ij})^\top$, where $P_{ij}$ is the row of the routing matrix~$P$ 
corresponding to the index~$(i,j)$. For each class~$(i,j)$ we thus have an 
i.i.d.\ sequence of routing vectors $\{\phi^{ij}(n) \colon n \in \N\}$. Define 
the \emph{cumulative routing process} as $\Phi^{ij}(\cdot) := 
\{\Phi^{ij}(n)\colon n \in \Z_+\}$, where $\Phi^{ij} (0) := 0$ and $\Phi^{ij} 
(n) := \sum_{k=1}^n \phi^{ij}(k)$ for~$n\geq 1$. With this definition, 
$\Phi^{ij}_{k\ell}(n)$ counts the number of class~$(i,j)$ customers that have 
been routed to the queue of class~$(k,\ell)$ customers after the first~$n$ 
service completions of class~$(i,j)$ customers.

For each class~$(i,j)$, the stochastic processes $A_{ij}(\cdot)$ and 
$\Phi^{k\ell}_{ij}(\cdot)$, $T^x_{k\ell}(\cdot)$ for all~$(k,\ell) \in \cK$ 
completely determine the paths of the queue length process~$Q^x_{ij}(\cdot)$. 
To be precise, for each time~$t$ we can calculate~$Q^x_{ij} (t)$ by adding the 
exogenous arrivals and the arrivals via routing to the initial queue length, 
and subtracting the departures. Thus we have for all~$t\in\Z_+$
\begin{align}
	& Q_{ij}^x(t) = Q_{ij}^x(0) + A_{ij}(t) + \sum\nolimits_{(k,\ell) \in\cK}
		\Phi^{k\ell}_{ij} \bigl( T_{k\ell}^x(t) \bigr) - T_{ij}^x(t),
		\vphantom\sum \label{queueing eq 1} \\
	& \text{$Q_{ij}^x(t) \geq 0$
		and $Q^x_{ij}(t+1) - Q^x_{ij}(t) \in \{-1,0,1\}$},
		\vphantom{\sum\nolimits_{(k,\ell)}} \label{queueing eq 2} 
		\displaybreak[1] \\
	& \text{$T_{ij}^x(t)$ is non-decreasing in~$t$,
		$T_{ij}^x(t+1)-T_{ij}^x(t) \leq 1$, and $T_{ij}^x(0) = 0$.}
		\vphantom\sum \label{queueing eq 3}
\end{align}

The multiclass network has the \emph{work-conserving property}, which means 
that a station will not idle whenever at least one of its queues is non-empty. 
To express this property in a formula, for $i \in \{1,\ldots,L\}$ and $t \in 
\Z_+$, let $I^x_i(t)$ denote the cumulative amount of time that station~$i$ 
has been idle up to time~$t$, and write $\Delta I^x_i (t) := I^x_i (t+1) - 
I^x_i (t)$. Then
\begin{align}
	& \text{$I_i^x(t) := t - \sum_{j=0}^L T_{ij}^x(t)$
		is non-decreasing in~$t$,}
		\vphantom\sum \label{queueing eq 4} \\
	& \text{$\sum_{t=0}^\infty \sum_{j = 0}^L Q_{ij}^x(t) \, \Delta I_i^x(t)
		= 0$ for all $i \in \{1,\ldots,L\}$}.
		\vphantom\sum \label{queueing eq 5}
\end{align}
Similarly to the way in which~\eqref{queueing eq 5} captures the 
work-conserving property, we can also account for the priorities in the 
system. Recall that we imposed that customers of classes~$(i,j)$ with $j \neq 
0$ have priority over class~$(i,0)$ customers. If we let $I^x_{i0} (t) := t - 
\sum_{j=1}^L T^x_{ij} (t)$ be the total service time available for 
class~$(i,0)$ customers up to time~$t$ at station~$i$, and write $\Delta 
I^x_{i0} (t) := I^x_{i0} (t+1) - I^x_{i0} (t)$, then the priority structure is 
captured by the \emph{priority equation}
\begin{equation}
	\text{$\sum_{t=0}^\infty \sum_{j=1}^L Q^x_{ij}(t) \, \Delta I^x_{i0} (t)
		= 0$ for all $i \in \{1,\ldots,L\}$}.
		\vphantom\sum \label{queueing eq 6}
\end{equation}

For each sample path of the multiclass network, the associated processes 
$Q^x_{ij} (\cdot)$ and~$T^x_{ij} (\cdot)$ have to satisfy 
equations~\eqref{queueing eq 1}--\eqref{queueing eq 6}. We refer to this set 
of equations as the \emph{queueing equations} or the \emph{queue length 
process representation}. The evolution of the sample paths clearly depends on 
the parameter setting, i.e., these equations are parameterized by the 
collection of $p_i$ and~$q_{ij}$.

\subsection{The fluid model}
\label{sec: fluid model}

We are now ready to introduce the fluid model equations, or \emph{fluid model} 
for short. They are continuous-time analogs of the queueing equations that 
arise naturally by taking a scaling limit of solutions to the queueing 
equations. To introduce these fluid model equations, it is convenient to first 
extend all the discrete-time processes introduced in Section~\ref{sec: 
queueing equations} to continuous time by linear interpolation. That is, we 
define $Q_{ij}^x (t)$ for non-integer values of~$t \in \R_+$ by interpolating 
linearly between $Q^x_{ij} \bigl( \floor{t} \bigr)$ and $Q^x_{ij} \bigl( 
\floor{t}+1 \bigr)$, and we extend the other processes introduced in 
Section~\ref{sec: queueing equations} to continuous time in the same way.%
\footnote{In continuous time, $Q^x_{ij} (t)$ can be interpreted as the fluid 
mass present in the class~$(i,j)$ queue at time~$t$ if we replace customers by 
unit masses of fluid that flows in and out of the queues at unit rate, 
controlled by valves at the stations that open and close (according to the 
priority structure) at integer times.}
Furthermore, we switch to a vector notation by defining $Q^x (\cdot) := \{ 
Q^x(t) \colon t\in \R_+\}$ as the vector-valued process with 
components~$Q^x_{ij} (\cdot)$, $(i,j) \in \cK$. Likewise, we define~$T^x 
(\cdot)$ as the process with components~$T^x_{ij} (\cdot)$. All vectors are 
column vectors, and (in)equalities between vectors have to be read 
component-wise.

To obtain the fluid model, we scale the queue length and service time 
processes by the norm of the initial state~$x$. That is, we introduce the 
scaled processes $SQ^x (\cdot)$ and~$ST^x (\cdot)$ by setting
\[
	SQ^x (t) := \frac{\;\:1}{\norm{x} \bmax1} \, Q^x\bigl( \norm{x} t \bigr)
	\quad\text{and}\quad
	ST^x (t) := \frac{\;\:1}{\norm{x} \bmax1} \, T^x\bigl( \norm{x} t \bigr)
\]
for each $t \in \R_+$. We claim that for any sequence of initial states with 
norm tending to~$\infty$, these scaled processes converge at almost every 
sample point along a subsequence to limit processes $\bQ(\cdot)$ 
and~$\bT(\cdot)$ which we refer to as a \emph{fluid limit}. Moreover, this 
fluid limit has to satisfy fluid model analogs of the queueing 
equations~\eqref{queueing eq 1}--\eqref{queueing eq 6}. The first three of 
those equations are
\begin{align}
	& \bQ(t) = \bQ(0) + p\, t  - (I-P^\top) \, \bT(t),
		\vphantom\sum \label{fluid eq 1} \\
	& \text{$\bQ(t) \geq 0$
		and each component of~$\bQ(\cdot)$ is Lipschitz-1},
		\vphantom\sum \label{fluid eq 2} \displaybreak[1] \\
	& \text{$\bT(t)$ is non-decreasing in~$t$,
		each component of~$\bT(\cdot)$ is Lipschitz-1, and $\bT(0)=0$},
		\vphantom\sum \label{fluid eq 3}
\end{align}
where $I$ denotes the identity matrix of dimension~$L^2+L$ and $p$ is the 
vector with components~$p_{ij}$ for $(i,j) \in \cK$ defined by $p_{i0} := p_i$ 
and $p_{ij} := 0$ if $j \neq 0$. Next we have the two equations
\begin{align}
	& \text{$\bI(t) := e\, t - C \, \bT(t)$ is non-decreasing in~$t$},
		\vphantom\sum \label{fluid eq 4} \\
	& \text{$\int_0^\infty \bQ_{i0}(t) \diff \bI_i(t) = 0$
		for all $i \in \{1,\ldots,L\}$},
		\vphantom\sum \label{fluid eq 5}
\end{align}
where $e$ is the $L$-dimensional vector of ones and $C$ is the incidence 
matrix linking each station to the customer classes served by that station; it 
has dimension~$L \times (L^2+L)$ and entries defined for $i \in 
\{1,\dots,L\}$ and $(j,k) \in \cK$ by $C_{i,jk} = 1$ if $j=i$ and $C_{i,jk} = 
0$ otherwise. Finally,
\begin{equation}
	\text{$\bQ_{ij} (\cdot) \equiv 0$ for all $i,j \in \{1,\ldots,L\}$}.
		\vphantom\sum \label{fluid eq 6}
\end{equation}
It may not be obvious that this last equation captures the priority structure 
in the fluid limit, but note that \eqref{fluid eq 6} together 
with~\eqref{fluid eq 1} yields $\bT_{ij}(t) = \sum_{(k,\ell) \in \cK} 
\bT_{k\ell}(t) \, P_{k\ell,ij}$ for $i,j \in \{1,\dots,L\}$. This says that in 
the fluid limit, the stations must instantly spend service time on customers 
who enter the ring to route them through the network, reflecting their 
priority.

We refer to equations \eqref{fluid eq 1}--\eqref{fluid eq 6} as the 
\emph{fluid model equations} or \emph{fluid model}. Every pair of 
functions~$\bQ(\cdot), \bT(\cdot)$ that is a solution to these equations is 
called a \emph{fluid model solution}. Our claims about convergence to a fluid 
limit which is a fluid model solution are made precise in the following 
theorem. A general version of this result has been proven in~\cite{dai} for 
continuous-time multiclass networks. Because of the discrete-time nature of 
our processes, the proof of the theorem requires some small modifications, and 
is therefore given in Appendix~\ref{sec: appendix convergence fluid limit}.

\begin{theorem}\label{Thm: main thm convergence fluid limit}
	There exists a set of sample points $\Omega' \subset \Omega$ with 
	$\PP(\Omega') = 1$ such that for every sequence of initial states~$\{ x'_n 
	\}$ with $\norm{x'_n} \to \infty$, we can find for each~$\omega \in 
	\Omega'$ a subsequence~$\{ x''_n (\omega) \}$ and functions $\bQ (\omega, 
	\cdot), \bT (\omega, \cdot)$ so that everywhere on~$\Omega'$ (omitting 
	$\omega$ from the notation),
	\begin{equation}
		\sup_{t\in[0,m]} \, \normbig{ SQ^{x''_n}(t) - \bQ(t) }
		\bmax \normbig{ ST^{x''_n}(t) - \bT(t) } \to 0
		\qquad \text{as $n\to\infty$}
		\label{main thm convergence to fluid limit}
	\end{equation}
	for all~$m\in\N$. Moreover, on~$\Omega'$, the pair $\bQ (\cdot), \bT 
	(\cdot)$ is a fluid model solution with $\norm{\bQ(0)} = 1$.
\end{theorem}

\subsection{Stability of the fluid model}
\label{sec: stability}

In this section we complete the proof of Theorem~\ref{Thm: (main) stability of 
the model}, with an appeal to Lemma~\ref{Lemma: coupling}, by proving that the 
multiclass network defined in Section~\ref{sec: multiclass network 
description} is stable under the stability condition of the theorem. The key 
is to prove stability of the fluid model, which is defined as follows:

\begin{definition}\label{Def: fluid model stability}
	We say that the fluid model~\eqref{fluid eq 1}--\eqref{fluid eq 6} is 
	\emph{stable} if there exists a constant $\delta > 0$ that depends only on 
	$p$ and~$P$ such that any fluid model solution $\bQ(\cdot), \bT(\cdot)$ 
	with $\norm{\bQ(0)} = 1$ satisfies $\bQ(t) = 0$ for all~$t\geq\delta$.
\end{definition}

It is well-known that in general, stability of the fluid model is a sufficient 
condition for stability of~$X^2(\cdot)$. This result has been proven for 
Markov chains with general state spaces under an additional condition, e.g., 
see \cite[Theorem~4.2] {dai} or~\cite[Theorem~4.16] {bramson}. In the setting 
of a countable state space, a considerably shorter argument is given 
in~\cite{bramson2}. For completeness, we formulate this key result in the 
following proposition, a proof of which is included in Appendix~\ref{sec: 
appendix convergence fluid limit}.

\begin{proposition}\label{Prop: original Markov chain pos rec if fluid limit 
	model is stable}
	If the fluid model is stable, then the Markov chain~$X^2(\cdot)$ is 
	positive recurrent.
\end{proposition}

So all that remains is to prove that, under the stability condition in 
Theorem~\ref{Thm: (main) stability of the model}, the fluid model is stable. A 
more general version of this result is proven in~\cite[Section~6] {dai1996}, 
but here we present a more transparent proof based on~\cite{tassiulas1996}. It 
is shown in~\cite{tassiulas1996} that under deterministic assumptions on the 
external arrivals and routing in a ring network, any work-conserving policy 
stabilizes the ring. These assumptions are satisfied in our case because in 
the fluid limit, the arrival and routing processes become deterministic due to 
the strong law of large numbers (this is one of the key steps in the proof of 
Theorem~\ref{Thm: main thm convergence fluid limit} in Appendix~\ref{sec: 
appendix convergence fluid limit}).

In our proof that the fluid model is stable, we consider the residual amount 
of time stations will be busy, given the amount of fluid mass that is present 
in the network. In essence, the goal of the proof is to show that this 
quantity, and hence the total fluid mass, must become zero before a fixed 
deterministic time that is determined only by the parameter setting. This 
means that we require a stronger result than the main theorem 
in~\cite{tassiulas1996}, which (translated to our context) states only that 
the total fluid mass goes to zero in the limit~$t\to \infty$. Moreover, 
traffic in~\cite{tassiulas1996} leaves the ring as soon as it reaches its 
destination, whereas fluid mass in our model can complete multiple full 
circles (albeit instantly) before being removed. This means that the 
quantities we consider, such as the residual work, are slightly different than 
the ones used in~\cite{tassiulas1996}, and we also need to adapt the argument 
to obtain the stronger result we require.

It is customary in the theory of multiclass queueing networks to consider the 
quantities
\[
	\lambda := \bigl( I-P^\top \bigr)^{-1} p
	\qquad\text{and}\qquad
	\rho := C \lambda.
\]
The quantity~$\lambda$ is known as the solution of the traffic 
equations~\cite[Section~1.2] {bramson}, and $\lambda_{ij}$ can be interpreted 
as the effective arrival rate to the queue of class~$(i,j)$ customers. By 
summing over~$j$ we obtain the effective arrival rate at station~$i$, given by 
$\rho_i = (C\lambda)_i = \sum_{j=0}^L \lambda_{ij}$. In any multiclass 
network, $\rho_i$ represents the amount of service time that arrives to the 
network per unit time and is required from station~$i$ in the future. One can 
prove that the multiclass network cannot be stable when $\rho<1$ fails to 
hold, using an argument similar to the one we used to prove 
Proposition~\ref{Prop: necessity main theorem}. In many cases $\rho < 1$ is 
sufficient for stability, although this is not always true (see 
\cite[Chapter~3] {bramson} for counter-examples). The next lemma establishes 
that the condition for stability in Theorem~\ref{Thm: (main) stability of the 
model} is equivalent to the condition that~$\rho<1$. 

\begin{lemma}\label{Lemma: rho < 1 iff pi_i0 > p_i}
	For all~$i \in \{1,\dots,L\}$ we have $\rho_i = p_i + 1 - \pi_{i0}$ and 
	hence
	\[
		\rho_i < 1 \quad \text{if and only if} \quad \pi_{i0} > p_i.
	\]
\end{lemma}

\begin{proof}
	Let $i \in \{1,\ldots,L\}$. We will show that $\rho_i = p_i + \sum_{j=1}^L 
	b_{ij} \, p_j$, where we recall from Section~\ref{sec: marg stat dist} 
	that $b_{ij}$ was defined in the roundabout model as the expected number 
	of visits to cell~$i$ by a type-$j$ vehicle that enters the roundabout. We 
	start by observing that
	\[
		\rho
		= C \bigl( I-P^\top \bigr)^{-1} p
		= C \biggl( I + \sum_{m = 1}^\infty \bigl( P^\top \bigr)^m
			\biggr)\, p.
	\]
	Hence, using the definition~\eqref{Eqn: routing matrix} of the routing 
	matrix~$P$ and the fact that $p_{i0} = p_i$ and $p_{ij} = 0$ for $j\neq0$, 
	we obtain
	\[
		\rho_i
		= \sum_{j=0}^L p_{ij} + \sum_{j=0}^L \sum_{m=1}^\infty
			\sum_{(k,\ell) \in \cK} (P^m)_{k\ell,ij} \, p_{k\ell}
		= p_i + \sum_{j=1}^L \sum_{m=1}^\infty (P^m)_{j0,ij} \, p_j.
	\]
	Here, $(P^m)_{j0,ij}$ is the probability that a class~$(j,0)$ customer is 
	routed to become a class~$(i,j)$ customer in $m$~steps, which is the same 
	as the probability in the roundabout model that a vehicle that enters the 
	roundabout at cell~$j$ visits cell~$i$ after having stayed on the 
	roundabout for $m-\nobreak 1$ time steps. Hence, we see that $\rho_i$ is 
	equal to $p_i + \sum_{j=1}^L b_{ij} \, p_j$, and we are done.
\end{proof}

The proof that the fluid model is stable hinges on another crucial lemma 
concerning the residual amount of work for each station~$i$ at a given time. 
We claim that this quantity is described by the process $\bR (\cdot) := 
\bigl\{ \bR(t) \colon t \in \R_+ \bigr\}$ defined through
\begin{equation}
	\label{Def: residual work}
	\bR(t) := C\bigl( I-P^\top \bigr)^{-1} \bQ(t).
\end{equation}
To see that this process describes the residual work in the system, note first 
of all the analogy with the definition of~$\rho$ as $C \bigl( I-P^\top 
\bigr)^{-1} p$. Since we have both $p_{ij} = 0$ and, by~\eqref{fluid eq 6}, 
$\bQ_{ij} (t) = 0$ for~$j\neq 0$, we can repeat the steps in the proof of 
Lemma~\ref{Lemma: rho < 1 iff pi_i0 > p_i} with $\bQ(t)$ in place of~$p$ to 
see that
\begin{equation}
	\label{Eqn: R_i}
	\bR_i(t)
	= \bQ_{i0}(t) + \sum_{j=1}^L b_{ij} \, \bQ_{j0}(t).
\end{equation}
By definition, $b_{ij}$ is the mean number of times a customer who arrives at 
station~$j$ in the multiclass network (i.e., a class~$(j,0)$ customer) intends 
to visit station~$i$ (as a class~$(i,j)$ customer) in the future. It follows 
that we can indeed interpret~$\bR_i(t)$ as the residual amount of work for 
station~$i$ that is present in the system at time~$t$.

This quantity has an important property, which is the fluid equivalent of the 
statement that in the multiclass network, when the queue of class~$(i,0)$ 
customers is empty, all customers who are present in the system and will visit 
station~$i$, must also visit the preceding station~$i-\nobreak1$. This 
crucially reflects the ring topology in our system, and enables us to control 
the decay of the amount of fluid in the fluid model. We formulate this key 
result in the following lemma:

\begin{lemma}\label{Lemma: circularity bound R_i}
	If $\bQ(\cdot), \bT(\cdot)$ is a fluid model solution, then for all~$i \in 
	\{1,\ldots,L\}$ and $t \geq 0$,
	\[
		\bQ_{i0}(t) = 0
		\quad\text{implies}\quad
		\bR_{i-1}(t) \geq \bR_{i}(t),
	\]
	where $i-1$ has to be read as $L$ when $i=1$.
\end{lemma}

\begin{proof}
	Let $t \geq 0$ be arbitrary and assume $\bQ_{i0} (t) = 0$. Let $k := L - 
	(L +\nobreak 1 -\nobreak i \mod L)$ denote the station preceding 
	station~$i$ in the network. Then, using~\eqref{Eqn: R_i}, we have
	\[\begin{split}
		\bR_k(t) - \bR_i(t)
		&= \bQ_{k0}(t) + \sum_{j=1}^L \sum_{m=1}^\infty (P^m)_{j0,kj} \, 
		\bQ_{j0}(t) - \sum_{j=1}^L \sum_{m=1}^\infty (P^m)_{j0,ij} \, 
		\bQ_{j0}(t) \\
		&= \bQ_{k0}(t) - \sum_{j=1}^L P_{j0,ij} \, \bQ_{j0}(t) + \sum_{j=1}^L 
		\sum_{m=1}^\infty \bigl( (P^m)_{j0,kj} - (P^{m+1})_{j0,ij} \bigr) \, 
		\bQ_{j0}(t).
	\end{split}\]
	Since $P_{k0,ik} = 1$ and $P_{j0,ij} = 0$ for $j\neq k$, the first two 
	terms on the right cancel each other. Furthermore, since for $m\geq1$, a 
	class~$(j,0)$ customer who is routed to become a class~$(i,j)$ customer in 
	$m+1$~steps must first be routed to become a class~$(k,j)$ customer after 
	$m$~steps, each term under the sum over~$m$ is non-negative. Hence, 
	$\bR_k(t) \geq \bR_i(t)$.
\end{proof}

We are now ready to state and prove the main result of this section, the 
stability of the fluid model under the condition~$\rho<1$ from 
Lemma~\ref{Lemma: rho < 1 iff pi_i0 > p_i}, which also completes the proof of 
Theorem~\ref{Thm: (main) stability of the model}.

\begin{theorem}\label{Thm: stability fluid limit model}
	If $\rho_i < 1$ for each $i \in \{1,\dots,L\}$, then the fluid model 
	\eqref{fluid eq 1}--\eqref{fluid eq 6} is stable.
\end{theorem}

\begin{proof}
	Assume $\max_i\rho_i < 1$ and let the pair~$\bQ(\cdot), \bT(\cdot)$ be a 
	solution of the fluid model equations with $\norm{\bQ(0)} = 1$. We first 
	show that for this solution, the residual work for any station~$i$ 
	decreases at rate~$1-\rho_i$ during time intervals in which the 
	class~$(i,0)$ queue is continuously non-empty. To see this, note that 
	combining \eqref{Def: residual work} and~\eqref{fluid eq 1} gives $\bR(t) 
	= \bR(0) + \rho \, t - C \, \bT(t)$. So by~\eqref{fluid eq 4}, we can 
	express~$\bR_i(t)$ in terms of the cumulative idle time process of 
	station~$i$ as
	\begin{equation}\label{Eqn: change in R_i}
		\bR_i(t)
		= \bR_i(0) - (1-\rho_i) \, t + \bI_i(t).
	\end{equation}
	Since by the work-conserving property~\eqref{fluid eq 5} the cumulative 
	idle time process of station~$i$ cannot increase while the class~$(i,0)$ 
	queue is non-empty, it follows that
	\begin{equation}\label{Eqn: R_i decreases}
		\bR_i(s) - \bR_i(t) = (1-\rho_i) \, (t-s)
		\quad \text{if $\bQ_{i0}(\cdot)>0$ on~$(s,t\,]$.}
	\end{equation}

	Now suppose that for some $i_0 \in \{1,\dots,L\}$, $t_0>0$ 
	and~$\varepsilon > 0$ we have $\bQ_{i_00}(t_0) = \varepsilon$. Then 
	by~\eqref{Eqn: R_i} we have $\bR_{i_0}(t_0) \geq \varepsilon$, and using 
	that $\norm{\bQ(0)} = 1$ we also obtain $\max_i \bR_i(0) \leq (1+B)$, 
	where $B := \max_{ij} b_{ij}$. We claim that \eqref{Eqn: R_i decreases} 
	and Lemma~\ref{Lemma: circularity bound R_i} together imply
	\begin{equation}\label{Eqn: residual work time 0}
		\max\nolimits_i \bR_i(0)
		\geq \bR_{i_0}(t_0) + (1-\bar\rho) \, t_0
		= \varepsilon + (1-\bar\rho) \, t_0,
	\end{equation}
	where $\bar\rho := \max_i \rho_i < 1$. This yields $t_0 < \delta := (1+B) 
	(1-\bar\rho)^{-1}$, which proves that the fluid model satisfies 
	Definition~\ref{Def: fluid model stability} of stability with $\delta$ as 
	specified.

	It remains to prove~\eqref{Eqn: residual work time 0}. To this end, we 
	define
	\[
		t_1
		:= \sup\{ s\in [0,t_0] \colon \text{$s=0$ or $\bQ_{i_00}(s) = 0$} \},
	\]
	so that $(t_1,t_0]$ is the interval of maximal length ending at time~$t_0$ 
	during which the class~$(i_0,0)$ queue is continuously non-empty. Then 
	by~\eqref{Eqn: R_i decreases} and the fact that $\bR_{i_0}(t_0) \geq 
	\varepsilon$,
	\[
		\bR_{i_0}(t_1) \geq \varepsilon + (1-\bar\rho) \, (t_0-t_1).
	\]
	If $t_1=0$ this gives~\eqref{Eqn: residual work time 0}, and we are done. 
	If $t_1>0$ we proceed recursively, as we explain next.

	Suppose that for some~$n\geq1$, we have found an index~$i_{n-1}$ and 
	time~$t_n$ such that
	\[
		0<t_n<t_0, \quad
		\bQ_{i_{n-1}0}(t_n) = 0, \quad \text{and} \quad
		\bR_{i_{n-1}}(t_n) \geq \varepsilon + (1-\bar\rho) \, (t_0-t_n).
	\]
	Then at least one queue must be non-empty at time~$t_n$. We define~$i_n$ 
	as the index of the station nearest to, but before~$i_{n-1}$ on the ring, 
	such that~$\bQ_{i_n0}(t_n) > 0$; to be precise, we set $i_n := i$, where 
	$i$ is the index such that $\bQ_{i0}(t_n) > 0$ for which $(i_n + L - i) 
	\mod L$ is minimal. By Lemma~\ref{Lemma: circularity bound R_i} we then 
	have $\bR_{i_n} (t_n) \geq \bR_{i_{n-1}} (t_n)$. Analogously to how we 
	defined~$t_1$, we now set
	\[
		t_{n+1}
		:= \sup\{ s\in [0,t_n] \colon
				\text{$s=0$ or $\bQ_{i_n0}(s) = 0$} \},
	\]
	so that $\bQ_{i_n0} (\cdot) > 0$ on the time interval~$(t_{n+1},t_n]$. 
	Using~\eqref{Eqn: R_i decreases} we then obtain
	\[
		\bR_{i_n}(t_{n+1})
		\geq \bR_{i_n}(t_n) + (1-\bar\rho) \, (t_n-t_{n+1})
		\geq \varepsilon + (1-\bar\rho) \, (t_0-t_{n+1}),
	\]
	and we conclude that~\eqref{Eqn: residual work time 0} follows by 
	induction in~$n$, provided that $t_m = 0$ for some~$m\geq1$.

	To show this is indeed the case, note that we have $\bR_{i_n} (t_n) \geq 
	\varepsilon$ if $t_n>0$. It then follows from~\eqref{Eqn: R_i} that 
	$\bQ_{i0} (t_n) \geq \tau$ for some station~$i$, where $\tau := 
	\varepsilon /\nobreak (1+LB)$. But $\bQ_{i0} (\cdot)$ is Lipschitz-1 
	by~\eqref{fluid eq 2}, so $\bQ_{i0} (s)$ must be strictly positive for 
	all~$s$ in $( 0 \bmax (t_n - \tau), t_n]$. This implies that it cannot 
	possibly be the case that $t_{n+k} > 0 \bmax (t_n-\tau)$ for $k = 1,2, 
	\dots, L$. Because this is true for every~$n\geq1$ such that $t_n>0$, 
	$t_m$ must be~$0$ for some~$m\geq1$. This completes the proof.
\end{proof}

\section{Related models}
\label{sec: related models}

In the introduction we mentioned that our main result for the roundabout 
model, Theorem~\ref{Thm: (main) stability of the model}, also applies to the 
slotted-ring model in~\cite{vArem1990}, and proves a form of stability for 
that model under the stability condition. We explain this in more detail in 
the first part of this section. In the second part, we discuss how the model 
we introduced in Section~\ref{sec:model} differs from the roundabout model 
in~\cite{storm2020}, and provide a rigorous argument to show that this 
difference does not alter the global stability region. In the last part, we 
present ways in which the routing can be modified without affecting the 
validity of the proof of stability. In particular, we explain how our setting 
can also be used to cover formulations of multiclass networks with fixed 
routes.

\subsection{Slotted-ring model}
\label{subsec: slotted-ring model}

The slotted-ring model studied in~\cite{vArem1990} is a model in continuous 
time. It consists of a slotted ring, containing $c$~slots of equal length, 
that rotates with a constant rotation time~$\tau$, and $n$~stations located at 
fixed but arbitrary points on the ring. Packets arrive from outside to a queue 
at each station~$i$, and have a random destination station~$j$ chosen 
according to a fixed distribution. When an empty slot arrives at a station, 
the station can transmit a packet to the slot. This packet is removed from the 
slot as soon as at it reaches its destination. Due to the nature of the 
slotted-ring model, the stochastic process describing the model is in general 
not Markovian and cannot have a stationary limit distribution. For this 
reason, another form of stability called \emph{$\tau$-stability} was 
considered in~\cite{vArem1990}, which is defined as positive recurrence of the 
discrete-time Markov chain obtained by observing the process at the times~$t 
\,\tau$, with $t\in\Z_+$.

A crucial claim in~\cite{vArem1990} is that only the relative order of the 
stations, but not their exact positions, is relevant for $\tau$-stability. 
This freedom allows us to map to slotted-ring model to our roundabout model, 
as follows. Take $\tau/c$ as the unit of time, suppose that the times between 
arrivals of packets at station~$i$ are geometrically distributed with 
parameter~$p_i$, and assume there are no more stations than slots, i.e., 
$n\leq c$. Place the stations on the ring so that for $i=1,2,\dots,n-1$, the 
distance between stations $i$ and~$i+1$ is exactly the length of a slot. For 
the roundabout model, take $L=c$ and identify each slot with a cell. Let cells 
1 through~$n$ have external arrival rates $p_1,\dots,p_n$, and if $c>n$, set 
$p_{n+1},\dots,p_c$ equal to zero. Choose the parameters~$q_{ij}$ in the 
roundabout model so that the routing of packets in the slotted-ring model is 
reproduced (this requires setting certain~$q_{ij}$ equal to~1). Then the 
slotted-ring model observed at times~$t \mskip2mu \tau$ is equivalent to the 
roundabout model observed at times~$tL$, $t\in\Z_+$. It follows that the 
slotted-ring model is $\tau$-stable if and only if the roundabout model is 
stable.

The only issue that remains is that arrivals in the slotted-ring model 
in~\cite{vArem1990} were originally assumed to follow Poisson processes. This 
implies that there can be more than one arrival to each external queue in each 
unit of time, which we did not allow in the roundabout model. However, because 
of the strong law of large numbers, this does not fundamentally change the 
condition for stability. That is, suppose we allow the number of arrivals to 
queue~$i$ in one unit of time to follow a Poisson distribution with 
intensity~$\lambda_i$. Then the corresponding arrival process still converges 
in the fluid limit to the function $t \mapsto p_i\,t$, where $p_i = 
\lambda_i$. Therefore, as follows from the proof of Theorem~\ref{Thm: main thm 
convergence fluid limit}, we obtain the same fluid model, and hence the same 
stability condition, as for geometrically distributed interarrival times. We 
conclude that in the case $n\leq c$, Theorem~\ref{Thm: (main) stability of the 
model} proves $\tau$-stability of the slotted-ring model under the stability 
condition, without requiring the additional assumption (Assumption~1) that was 
made in~\cite{vArem1990}.

We now turn to the case~$n>c$. This case is more difficult, but we can handle 
it as follows. In the slotted-ring model, we still take $\tau/c$ as the unit 
of time, and we assume a Poisson arrival rate~$\lambda_i$ at station~$i$. We 
place the stations on the ring so that all distances between neighboring 
stations are equal. For the roundabout model, we take the number of cells~$L$ 
to be the least common multiple of $n$ and~$c$. We define $m := L/c$ and $k := 
L/n$, let $J$ be the set of indices $\{m, 2m, \dots, cm\}$, and let $K$ be the 
set of indices $\{k, 2k, \dots, nk\}$. The idea is that the positions of the 
slots at time~0 correspond to the cells with an index in~$J$, and the 
locations of the stations correspond to the cells with indices in~$K$. We 
therefore set the external arrival rates to $p_i := \lambda_{i/k} / m$ for 
$i\in K$ and $p_i := 0$ for $i\notin K$, and we assume the external queues at 
the cells with an index not in~$K$ are empty at time~0 (and hence at all 
times).

The delicate part is to reproduce the correct routing from the slotted-ring 
model in the roundabout model. To achieve this, we further assume that in the 
initial state each cell~$i$ that does not correspond to a slot (i.e., with 
$i\notin J$) is occupied by a vehicle of a type~$j$ not in~$K$, and that these 
types of vehicle never leave the system (that is, they are simply forwarded to 
the next cell at each time step). This has as a consequence that $b_{ij}$ is 
infinite for~$j\notin K$, but this does not matter for the analysis because 
$p_j=0$. With these restrictions in place, the parameters~$q_{ij}$ of the 
roundabout model are zero for $j\notin K$, and we can choose the remaining 
parameters~$q_{ij}$ such that the routing is the same as in the slotted-ring 
model. Crucially, the slotted-ring model observed at times~$t \mskip2mu \tau$ 
is then again equivalent to the roundabout model observed at times~$tL$ ($t\in 
\Z_+$), so that $\tau$-stability of the former model is equivalent to 
stability of the latter.

It is not difficult to see that with this setup, taking the fluid limit of the 
multiclass network associated with the roundabout model leads to the same 
fluid model as before. However, there are two additional constraints. First, 
every fluid limit has the property that $\bQ_{j0}(\cdot) \equiv 0$ for 
$j\notin K$. Second, because an exact fraction $1-m^{-1}$ of the cells in the 
roundabout model is continuously occupied by vehicles of types not in~$K$, 
every fluid limit satisfies $\sum_{j\notin K} \bT_{ij}(t) = t-tm^{-1}$ for $i 
= 1,2,\dots,L$ and all times~$t$. In the proof of stability of the fluid 
model, we should therefore only consider fluid model solutions that satisfy 
both constraints.

Let us investigate what this means for the proof of Theorem~\ref{Thm: 
stability fluid limit model}. For a meaningful analysis, in the residual work 
processes we should now only take into account customers of the classes 
$(j,0)$ and~$(i,j)$ with $j\in K$. That is, the residual work process of 
station~$i$ in the multiclass network is now given by
\[
	\bR_i(t) = \bQ_{i0}(t) + \sum_{j\in K} b_{ij} \, \bQ_{j0}(t),
\]
and it follows from \eqref{fluid eq 1} and~\eqref{fluid eq 4} that~\eqref{Eqn: 
change in R_i} is replaced by
\[
	\bR_i(t) = \bR_i(0) - (m^{-1}-\rho_i) \, t + \bI_i(t).
\]
We conclude that the proof of Theorem~\ref{Thm: stability fluid limit model} 
still goes through, but we now obtain stability of the fluid model under the 
stability condition $\max_i\rho_i < m^{-1}$. But this is exactly the desired 
condition for $\tau$-stability of the slotted-ring model, because in terms of 
the~$\lambda_i$ it is equivalent to
\[
	\lambda_{i/k} < 1 - \sum_{j\in K} b_{ij} \, \lambda_{j/k}
	\quad \text{for all~$i\in K$}.
\]

\subsection{Roundabout model}
\label{subsec: roundabout model}

We now explain our reasons for considering a slightly different model than the 
one in~\cite{storm2020}. In the roundabout model in~\cite{storm2020}, a 
vehicle that arrives to an empty queue and cell at time~$t \in \Z_+$ enters 
the roundabout immediately, and is therefore not in the queue at time~$t$ or 
time~$t+1$. This poses a problem for the multiclass network formulation of the 
model: in the multiclass network, a vehicle must have been in a queue for one 
unit of time in order to receive service.

Since our proof relies on coupling the roundabout model to a multiclass 
network, we have slightly modified the model here so that an arriving vehicle 
enters a queue first, and is allowed to enter the roundabout no sooner than 
one time unit later. We emphasize that this model alteration does not affect 
the stability condition: Theorem~\ref{Thm: (main) stability of the model} also 
applies to the model in~\cite{storm2020}. This follows from the fact that the 
two models can be coupled in such a way that their sample paths stay close 
together at all times. The precise result is stated in the proposition below, 
the proof of which is presented in Appendix~\ref{sec: appendix equality markov 
chains}. We conclude that the Markov chain~$X^1(\cdot)$ is positive recurrent 
if and only if the Markov chain for the model in~\cite{storm2020} is positive 
recurrent. 

\begin{proposition}\label{Prop: coupling old new model}
	Let $\tilde X(\cdot) := \{ \tilde X(t)\colon t \in \Z_+\}$ denote the 
	discrete-time Markov chain associated with the model in~\cite{storm2020}. 
	There exists a coupling between $X^1(\cdot)$ and~$\tilde X(\cdot)$ in 
	which at all times the states of all cells are the same for the two 
	processes, while the difference in queue length is at most~1 for each 
	queue.
\end{proposition}

\subsection{Models with other types of routing}
\label{subsec: model extensions}

We designed the multiclass network in Section~\ref{sec: multiclass network 
description} to mimic the routing of vehicles in the roundabout model. We will 
refer to this type of routing as `probabilistic routing'. This type of routing 
enables us to identify customers classes by customers' location and type.

A frequently used alternative in queueing networks is to fix a set of 
(deterministic) routes that customers can take through a network, and identify 
each customer class with one of these routes (see, e.g.,~\cite{bramson2}). 
These fixed routes are generally not included in the possible set of routes 
obtained from probabilistic routing, but our setting can be modified to handle 
this alternative type of routing. This is possible because in our setting, a 
route is determined by the location of the queue where the customer arrives 
together with the time the customer intends to spend in the network. That is, 
to replace probabilistic routing by fixed routes, we only need to change the 
distributions of the variables~$T_j$ we introduced in Section~\ref{sec: 
preliminaries}. This of course leads to different formulas for the marginal 
stationary probabilities~$\pi_{ij}$, but Lemma~\ref{Lemma: EN_ij} remains 
valid. To establish the pathwise coupling with a multiclass network, we need 
that only finitely many customer classes are required to describe all routes. 
In that case, the statement of Theorem~\ref{Thm: (main) stability of the 
model} and its proof go through, with a different expression for the 
$\pi_{ij}$, leading to a rigorous derivation of the global stability region.

Both formulations of the routing (and associated customer classes) have their 
advantages. On the one hand, fixed routes allow for customer behavior that is 
impossible with probabilistic routing. For instance, using fixed routes, we 
can have a class of customers who arrive to queue~$1$ and always complete 
exactly two full circles before leaving the roundabout at cell/station~$2$. 
This is not possible with probabilistic routing. On the other hand, with 
probabilistic routing, the time that customers can spend in the system is 
unbounded. With fixed (deterministic) routes, this time is necessarily 
bounded, since we can only handle a finite number of customer classes (i.e., 
routes of finite length) in the multiclass network.

A further model generalization is to combine probabilistic routing with fixed 
routes. That is, we could allow customers arriving at a station~$j$ to first 
follow a fixed route through the network (chosen from a finite set), and then 
either leave the network (with a certain probability), or continue according 
to probabilistic routing with parameters~$q_{ij}$. Again, this generalization 
only changes the distribution of the variable~$T_j$, so our methods still 
apply. But because the model formulation in Section~\ref{sec:model} already 
covers a broad range of applications in communication systems and 
transportation networks, we have chosen not to work with this more general 
setting.

\section{Concluding remarks}
\label{sec: conclusions}

In this paper, we have considered a ring-topology stochastic network with 
queues, involving different customer classes and a policy with a priority 
structure. A careful consideration of its marginal stationary distribution led 
us to a condition which we have proven to be necessary and sufficient for 
stability. In our proof, we coupled the sample paths of our model to those of 
a multiclass queueing network, enabling us to appeal to fluid model techniques 
to prove sufficiency of the stability condition. The approach we used 
explicitly exploited the relationship between the condition for stability and 
the rate at which different segments of the ring are occupied in case the 
model is stable and displays ergodic behavior.

Our proof of stability of the fluid model was based on~\cite{tassiulas1996}, 
where it was shown that \emph{any} work-conserving policy stabilizes the 
considered system. This raises the question whether our stability result holds 
under any work-conserving policy as well. We believe the answer is 
affirmative. To be more precise, without the priority structure we imposed, a 
fluid limit of our model in general may not satisfy equation~\eqref{fluid eq 
6}. This means we obtain a fluid model given by~\eqref{fluid eq 
1}--\eqref{fluid eq 5}, with $\bQ_{i0}(t)$ in equation~\eqref{fluid eq 5} 
replaced by~$\sum_{j=0}^L \bQ_{ij}(t)$. Proving stability of this fluid model 
then requires a similar replacement in Lemma~\ref{Lemma: circularity bound 
R_i} and the proof of Theorem~\ref{Thm: stability fluid limit model}, but we 
do believe that after making the appropriate modifications the proof will 
still go through.

With our main theorem on the model’s stability region, we can precisely 
quantify the capacity of associated communication or transportation systems. 
In practice, however, scenarios can occur in which the system is not stable. 
This often happens during certain periods such as rush hours in road traffic 
applications, or specific busy hours in communication networks. In that case, 
one is interested in determining for instance which queues will remain finite 
in length, and which ones will grow indefinitely, and at what rate, during 
these periods.

To address this problem, suppose we have $\rho_i > 1$ for some station~$i$. 
Then for any fluid model solution, $\bR_i(t)$ and hence also~$\norm{\bQ(t)}$ 
increases at least linearly in~$t$, by~\eqref{Eqn: change in R_i}. Following 
Dai~\cite{Dai1996a}, this is enough to conclude that in the multiclass 
network, with probability~1, $\lim_{t\to\infty} \norm{\bQ(t)} = \infty$. 
(Technically, we should consider a different fluid limit here, but it has the 
same fluid model except that~$\norm{\bQ(0)} = 0$; see~\cite{Dai1996a}.) It 
follows that if $\bar\rho := \max_i \rho_i > 1$, then our Markov chain is 
transient. Since we know the Markov chain is positive recurrent when 
$\bar\rho<1$, this only leaves the case~$\bar\rho=1$, which we believe is 
inaccessible with the methods we use here.

The next question we address is what we can say about the individual queues in 
the transient regime. Assuming ergodic behavior in the occupation of the 
cells, let $\tilde\pi_{i0}$ be the stationary probability at which cell~$i$ is 
empty. Then the ergodic rate at which type-$j$ vehicles enter the ring should 
be the smaller of $p_j$ and~$\tilde\pi_{j0}$, depending on whether queue~$j$ 
does not or does blow up, respectively. Following the arguments of 
Proposition~\ref{Prop: marginal distribution} and Lemma~\ref{Lemma: EN_ij}, 
the marginal stationary probability that cell~$i$ is occupied by a type-$j$ 
vehicle will then be given by $\tilde\pi_{ij} = b_{ij} \, (\tilde\pi_{j0} 
\bmin p_{j})$. Since $\sum_{j=0}^L \tilde\pi_{ij} = 1$ for every~$i$, it 
follows that the~$\tilde\pi_{i0}$ must satisfy
\begin{equation}\label{Eqn: unstable system}
	\tilde\pi_{i0}
	= 1 - \sum_{j=1}^L b_{ij} \, (\tilde\pi_{j0} \bmin p_j),
	\qquad i = 1,\dots,L.
\end{equation}
Therefore, solving this system of equations identifies the possible candidates 
for the ergodic behavior of our model in the transient regime.

To elaborate, suppose that for a given vector~$p$, \eqref{Eqn: unstable 
system} has a unique solution satisfying $p_i\neq \tilde\pi_{i0}$ for all~$i$ 
(again, we expect our methods cannot deal with boundary cases where $p_i = 
\tilde\pi_{i0}$ for some~$i$). Let $U$ and~$S$ be the sets of indices~$i$ for 
which $p_i > \tilde\pi_{i0}$ and $p_i < \tilde\pi_{i0}$, respectively. Note 
that by~\eqref{Eqn: pi_ij} and Lemma~\ref{Lemma: EN_ij}, the~$\pi_{i0}$ 
solve~\eqref{Eqn: unstable system} if the system is stable, so that we must 
have~$U\neq \varnothing$ in the transient regime. We call the queues in~$U$ 
unstable and the ones in~$S$ stable. The intuition behind this is that 
for~$i\in U$, we expect queue~$i$ to grow (eventually) at the asymptotic 
rate~$p_i-\tilde\pi_{i0}$. This means for our model that the unstable queues 
eventually always have a vehicle on offer to send onto the ring. Their precise 
length is therefore irrelevant in the long run, so we may as well omit these 
queues from our state space and describe the situation with a new Markov chain 
that only keeps track of the states of the stable queues and all the cells; we 
refer to~\cite{Adan2020} for the theory behind such a setup. We then expect 
this new Markov chain to be stable because $p_i < \tilde\pi_{i0}$ for 
all~$i\in S$, in analogy with our main Theorem~\ref{Thm: (main) stability of 
the model}.

We study the solution set of the system~\eqref{Eqn: unstable system} and its 
implications for the ergodic behavior of the queues in the unstable regime in 
a forthcoming paper~\cite{kager2022}.

\bibliographystyle{abbrv}
\bibliography{Stability}
\newpage

\begin{appendices}
\section{Equivalence of cellular automata and multiclass network formulation}
\label{sec: appendix equality markov chains}

This appendix is devoted to proving Lemma~\ref{Lemma: coupling} and 
Proposition~\ref{Prop: coupling old new model}. We start with the former.

\begin{proof}[Proof of Lemma~\ref{Lemma: coupling}]
	We have to provide a coupling between the Markov chains $X^1(\cdot)$ 
	and~$X^2(\cdot)$ and a bijection $f \colon \cS^1 \to \cS^2$ such that, for 
	every $\omega \in \Omega$, if $f\bigl( X^1(\omega,0) \bigr) = 
	X^2(\omega,0)$, then $f\bigl( X^1(\omega,t) \bigr) = X^2(\omega,t)$ for 
	every $t \in \Z_+$. The crux of the proof is that we use the same 
	collection of uniform random variables to construct the sample paths of 
	both processes explicitly, for given initial states $X^1(0)$ and~$X^2(0)$. 
	This establishes the coupling, after which we define the bijection~$f$ and 
	verify that the coupling has the desired property. Throughout the proof, 
	we suppose that we work with an arbitrary, but fixed set of parameters 
	$p_i$ and~$q_{ij}$, $i,j \in \{1,\ldots,L\}$.

	The first step is to construct the coupling of the two processes 
	$X^1(\cdot)$ and~$X^2(\cdot)$. For~$X^1(\cdot)$, we denote the states of 
	cell~$i$ and queue~$i$ at time~$t$ by $C^1_i(t)$ and~$Q^1_i(t)$, 
	respectively. Likewise, for~$X^2(\cdot)$ we denote the length of the queue 
	for class~$(i,j)$ customers at time~$t$ by $Q^2_{ij}(t)$.

	Next we introduce, for each $t \in \Z_+$, $L^2+L$ random variables~$U_{ij} 
	(t)$, where $(i,j) \in \cK$. We assume the~$U_{ij} (t)$ are all 
	independent and uniformly distributed on~$(0,1)$. The idea is that, for $i 
	\in \{1,\ldots,L\}$ and $t \in \Z_+$, the realization of the random 
	variable~$U_{i0}(t)$ determines whether a vehicle arrives at queue~$i$ in 
	the roundabout model, and also whether a class~$(i,0)$ customer arrives in 
	the multiclass network. Likewise, for $i,j \in \{1,\ldots,L\}$, the 
	realization of the random variable~$U_{ij}(t)$ determines whether a 
	type-$j$ vehicle departs from cell~$i$ (if present), and simultaneously 
	determines the routing of a class~$(i,j)$ customer.

	First, to construct the process~$X^1(\cdot)$, let the initial 
	state~$X^1(0) \in\nobreak \cS^1$ be given. Then it suffices to state how 
	the queues and cells have to be updated from one unit of time to the next 
	to uniquely determine the sample paths of the process. So let $t \geq 0$. 
	From cases~1--3 in Section~\ref{sec:model}, it is readily verified that 
	the following equations update the states of cell~$i +\nobreak 1$ and 
	queue~$i$ correctly, where $i+1$ has to be read as~1 in case $i=L$:
	\begin{align}
		& Q^1_i(t+1) = Q^1_i(t) - \one{Q^1_i(t)>0} \one{C^1_i(t)=0}
			+ \one{U_{i0}(t+1)\leq p_i},
			\label{eq: CA update queue}\\
		& C^1_{i+1}(t+1) = i\, \one{Q^1_i(t)>0} \one{C^1_i(t)=0}
			+ {\textstyle\sum_{j=1}^L} \, j\, \one{C^1_i(t) = j} 
			\one{U_{ij}(t) > q_{ij}}.
			\label{eq: CA update cell}
	\end{align}
	Indeed, these equations determine the states of each cell and queue in the 
	system, for each $t \geq 0$, in terms of $X^1(0)$ and the uniform random 
	variables~$U_{ij}(t)$.

	Similarly, to construct~$X^2(\cdot)$, again let the initial state~$X^2(0) 
	\in\nobreak \cS^2$ be given. Stating the one-step update rule is again 
	sufficient to determine the sample paths of~$X^2(\cdot)$. For $t \geq 0$, 
	from the definition of the multiclass network in Section~\ref{sec: 
	multiclass network description}, we see that taking
	\begin{align}
		& Q^2_{i0}(t+1) = Q^2_{i0}(t)
			- \one{Q^2_{i0}(t)>0} \one{\sum_{j=1}^L Q^2_{ij}(t) = 0}
			+ \one{U_{i0}(t+1) \leq p_i},
			\label{eq: QN update queue} \\
		& Q^2_{i+1,j}(t+1) = \delta_{ij} \, \one{Q^2_{i0}(t)>0}
			\one{\sum_{\ell=1}^L Q^2_{i\ell}(t)=0}
			+ Q^2_{ij}(t) \, \one{U_{ij}(t) > q_{ij}}
			\label{eq: QN update cell}
	\end{align}
	does the job. Here $i,j \in \{1,\ldots,L\}$ and $\delta_{ij}$ is the 
	Kronecker delta function, which is equal to one if $i = j$, and is zero 
	otherwise. This establishes the coupling between $X^1(\cdot)$ 
	and~$X^2(\cdot)$.

	The next step is to provide a bijection $f \colon \cS^1 \to \cS^2$ such 
	that $f\bigl( X^1(\omega,0) \bigr) = X^2(\omega,0)$ implies
	\begin{equation}
		\label{Eqn: bijective property coupling (proof)}
		f\bigl( X^1(\omega,t) \bigr) = X^2(\omega,t)
		\qquad \text{for all $t \in \Z_+$.}
	\end{equation}
	Such a bijection is given by the function~$f$ that maps a vector 
	in~$\cS^1$ component-wise to~$\cS^2$ with component functions $f_i \colon 
	\Z_+ \times\nobreak \{0,1,\dots,L\} \rightarrow \Z_+ \times\nobreak 
	\mathcal{V}$ for $i \in \{1,\dots,L\}$ given by
	\[
		(Q^1_i,C^1_i) \mapsto
		\bigl( Q^1_i, \one{C_i^1=1}, \one{C_i^1=2}, \dots, \one{C_i^1=L} 
		\bigr).
	\]
	In words, for each $i \in \{1,\ldots,L\}$, the function~$f_i$ maps the 
	states of queue and cell~$i$ from~$X^1(\cdot)$ to the states of all queues 
	at station~$i$ from~$X^2(\cdot)$. This function is a bijection since the 
	function $f^{-1} \colon \cS^2 \to \cS^1$, with component functions 
	$f_i^{-1} \colon \Z_+ \times\nobreak \mathcal{V} \to \Z_+ \times\nobreak 
	\{0,1,\dots,L\}$ given by
	\[
		\bigl( Q^2_{i0}, Q^2_{i1}, \dots, Q^2_{iL} \bigr) \mapsto
		\bigl( Q^2_{i0}, {\textstyle\sum_{j=1}^L} \, j \, Q^2_{ij} \bigr),
	\]
	is an inverse for~$f$. It is now readily verified that $f\bigl( 
	X^1(\omega,t) \bigr) = X^2(\omega,t)$ is equivalent to
	\[
		Q_i^1(\omega,t) = Q_{i0}^2(\omega,t)
		\quad\text{and}\quad
		\one{C_i^1(\omega,t)=j} = Q^2_{ij}(\omega,t)
		\qquad\text{for all $i,j \in \{1,\ldots,L\}$.}
	\]
	Therefore, it immediately follows from \eqref{eq: CA update 
	queue}--\eqref{eq: QN update cell} that for every $t \geq 0$,  if $f\bigl( 
	X^1(\omega,t) \bigr) = X^2(\omega,t)$, then $f\bigl( X^1(\omega,t+1) 
	\bigr) = X^2(\omega,t+1)$. Hence, if we assume $f\bigl( X^1(\omega,0) 
	\bigr) = X^2(\omega,0)$, then by induction we obtain~\eqref{Eqn: bijective 
	property coupling (proof)}, and the proof is complete.
\end{proof}

\begin{proof}[Proof of Proposition~\ref{Prop: coupling old new model}.]
	We begin by constructing the sample paths of~$\tilde X(\cdot)$ in terms of 
	the uniform random variables~$U_{ij}(t)$ from the previous proof of 
	Lemma~\ref{Lemma: coupling}. To this end, let $\tilde C_i(t)$ and~$\tilde 
	Q_i(t)$ denote, respectively, the state of cell~$i$ and length of 
	queue~$i$ at time $t \in \Z_+$ in the Markov chain~$\tilde X(\cdot)$. It 
	is then readily verified that for all $i \in \{1,\ldots,L\}$, the 
	following rules update the states of the cells and queues correctly and 
	determine the sample paths of~$\tilde X(\cdot)$, given the initial 
	state~$\tilde X(0)$ (as before, $i+1$ has to be read as~1 in case 
	$i=L$):
	\begin{align}
		& \tilde Q_i(t+1) = \tilde Q_i(t)
			- \one{\tilde Q_i(t)>0} \one{\tilde C_i(t)=0} \one{U_{i0}(t)>p_i}
			+ \one{\tilde C_i(t)>0} \one{U_{i0}(t)\leq p_i},
			\label{Eqn: old model queue} \\
		& \tilde C_{i+1}(t+1) = i \, \one{\tilde Q_i(t)>0}
			\one{\tilde C_i(t) = 0}
			+ i \,  \one{\tilde Q_i(t)=0} \one{\tilde C_i(t) = 0} 
			\one{U_{i0}(t) \leq p_i} \nonumber \\
		& \phantom{\tilde C_{i+1}(t+1) = \null}
			+ {\textstyle\sum_{j=1}^L} \, j \, \one{\tilde C_i(t)=j} 
			\one{U_{ij}(t)>q_{ij}}.
			\label{Eqn: old model cell}
	\end{align}

	As in the proof of Lemma~\ref{Lemma: coupling}, we have now coupled the 
	dynamics of $X^1(\cdot)$ and~$\tilde X(\cdot)$ in terms of the 
	variables~$U_{ij}(t)$. In addition, assuming the initial state~$\tilde 
	X(0)$ of~$\tilde X(\cdot)$ is given, we couple the initial states of the 
	two processes by setting $Q^1_i(0) := \tilde Q_i(0) + \one{U_{i0}(0) \leq 
	p_i}$ and $C^1_i(0) := \tilde C_i(0)$ for all $i \in \{1,\ldots,L\}$. We 
	claim that for all $t \in \Z_+$ and $i \in \{1,\ldots,L\}$,
	\begin{equation}
		\label{Eqn: coupling 2 property}
		Q^1_i(t) - \tilde Q_i(t) = \one{U_{i0}(t)\leq p_i}
		\quad\text{and}\quad
		C^1_i(t) - \tilde C_i(t) = 0,
	\end{equation}
	from which the result follows. We prove this claim with induction. 
	Clearly, \eqref{Eqn: coupling 2 property}~holds for $t = 0$ by the 
	coupling, so assume it holds for $t \in \Z_+$. Then, using the fact that 
	\eqref{Eqn: coupling 2 property} implies
	\[
		\one{Q^1_i(t)>0} = \one{\tilde Q_i(t)>0} + \one{\tilde Q_i(t)=0} 
		\one{U_{i0}(t)\leq p_i},
	\]
	it follows from \eqref{eq: CA update queue}--\eqref{eq: CA update cell} 
	and \eqref{Eqn: old model queue}--\eqref{Eqn: old model cell} that 
	\eqref{Eqn: coupling 2 property} also holds at time~$t+1$.
\end{proof}

\section{Fluid limits}
\label{sec: appendix convergence fluid limit}

In this appendix we prove Theorem~\ref{Thm: main thm convergence fluid limit} 
and Proposition~\ref{Prop: original Markov chain pos rec if fluid limit model 
is stable}. The proof of Theorem~\ref{Thm: main thm convergence fluid limit} 
requires two lemmas about sequences of functions that converge uniformly on a 
compact domain.

\begin{lemma}
	\label{Lemma: fLLN}
	If the function $f\colon \R_+ \to\nobreak \R_+$ is non-decreasing and for 
	some $a \in \R_+$, the sequence $\{ f(n)/n \colon n\in\N \}$ converges 
	to~$a$ as $n \to\nobreak \infty$, then $\sup_{t\in[0,m]} \abs{ f(nt)/n - 
	at } \to\nobreak 0$ for every~$m\in\N$.
\end{lemma}

\begin{proof}
	Choose a sequence $t_1, t_2, \ldots$ in~$\N$ that diverges to~$\infty$ but 
	so that $t_n/n \to 0$ as $n\to\infty$, and write $\delta_n := t_n/n$. 
	Since $f$ is non-decreasing, $\abs{ f(nt)/n - at }$ is bounded for all~$t 
	\in [0,\delta_n]$ by $( f(t_n)/t_n + a) \delta_n$. Moreover, since for 
	all~$k \in \N$ and all~$t$ in the interval~$\bigl[ k/n, (k+1)/n \bigr]$,
	\[
		\frac{k}{n} \Bigl( \frac{f(k)}{k} - a \Bigr) - \frac{a}{n}
		\leq \frac{f(nt)}{n} - at
		\leq \frac{k+1}{n} \Bigl( \frac{f(k+1)}{k+1} - a \Bigr)
		+ \frac{a}{n},
	\]
	$\abs{ f(nt)/n - at }$ is bounded for all~$t \in [\delta_n, m]$ by the 
	maximum of $m \, \abs{ f(k)/k-a } + a/n$ taken over all~$k \in \{t_n, 
	t_n+1, \dots, mn\}$. Hence, the result follows from $f(n)/n \to a$ and 
	$\delta_n \to 0$.
\end{proof}

\begin{lemma}
	\label{Lemma: uoc convergence of sequences}
	Let $m\in\N$, and let $\{f_n\}$ and~$\{g_n\}$ be sequences of Lipschitz-1 
	functions from~$\R_+$ to~$\R_+$ that converge uniformly on the 
	domain~$[0, m]$ to the Lipschitz-1 functions $f$ and~$g$, respectively, 
	where $g$ and all~$g_n$ are non-decreasing and satisfy $g(0) = g_n(0) = 
	0$. Then
	\begin{enumerate}
		\item[(a)] $\sup_{t\in[0,m]} \absbig{ f_n(g_n(t)) - f(g(t)) }\to 0$ as 
			$n\to\infty$;
		\item[(b)] $\int_0^m f_n(t)\diff g_n(t)\to \int_0^m f(t)\diff g(t)$ as 
			$n\to\infty$.
	\end{enumerate}
\end{lemma}

\begin{proof}
	Since $g(0) = 0$ and $g$ is Lipschitz-1, $t\in[0,m]$ implies $g(t) \in 
	[0,m]$. So, using that $f_n$ is Lipschitz-1 as well, it follows that 
	$\absbig{ f_n(g_n(t)) - f(g(t)) }$ is bounded for all~$t\in[0,m]$ by
	\[\begin{split}
		\absbig{ f_n(g_n(t)) - f_n(g(t)) }
			+ \absbig{ f_n(g(t)) - f(g(t)) }
		\leq \abs{ g_n(t)-g(t)}
			+ \sup\nolimits_{t\in[0,m]} \abs{ f_n(t) - f(t) },
	\end{split}\]
	from which we obtain~(a). To prove~(b), fix $\varepsilon > 0$. For any 
	simple function~$h\colon [0,m] \to\nobreak \R_+$, we can write the 
	difference between the Lebesgue--Stieltjes integrals $\int_0^m f_n \diff 
	g_n$ and~$\int_0^m f \diff g$ as
	\[\begin{split}
		\int_0^m (f_n-f) \diff g_n
			+ \int_0^m (f-h) \diff g_n
			+ \int_0^m (h-f) \diff g
			+ \int_0^m h\diff g_n - \int_0^m h\diff g.
	\end{split}\]
	We can choose~$n$ sufficiently large and, because $f$ is uniformly 
	continuous on~$[0,m]$, a simple function~$h$ so that $\abs{f_n-f} < 
	\varepsilon/m$ and $\abs{f-h} < \varepsilon/m$ uniformly on~$[0,m]$. Since 
	the functions $g_n$ and~$g$ are Lipschitz-1, it follows that each of the 
	first three integrals is bounded in absolute value by~$\varepsilon$. As 
	for the last two integrals, $\int_0^m h \diff g_n \to \int_0^m h\diff g$ 
	because $h$ is simple.
\end{proof}

\begin{proof}[Proof of Theorem~\ref{Thm: main thm convergence fluid limit}]
	We start with some observations. Recall that in Section~\ref{sec: fluid 
	model} we extended the discrete-time processes from Section~\ref{sec: 
	queueing equations} to continuous time by linear interpolation. Note that 
	in discrete time, the processes $Q^x_{ij} (\cdot)$ and~$T^x_{ij} (\cdot)$ 
	make steps of absolute size 0 or~1 only. It follows that their 
	continuous-time extensions are Lipschitz-1, and that the discrete-time 
	queueing equation~\eqref{queueing eq 1} in fact holds at all times~$t \in 
	\R_+$ for the continuous-time processes. But if a function~$f$ from~$\R_+$ 
	to~$\R_+$ is Lipschitz-1, then for any~$n \in \N$, so is the 
	function~$f_n$ defined by $f_n(t) := f(nt)/n$, since $\abs{ f(nt)-\nobreak 
	f(ns) } / n$ is bounded by~$\abs{t-\nobreak s}$ for all~$s,t \in \R_+$. 
	Hence, the scaled processes $SQ^x_{ij} (\cdot)$ and~$ST^x_{ij} (\cdot)$ 
	are also Lipschitz-1, and \eqref{queueing eq 1} implies
	\begin{equation}
		SQ_{ij}^x(t)
		= SQ_{ij}^x(0) + \frac1{\norm{x}} A_{ij}\bigl( \norm{x}t \bigr)
			+ \sum\nolimits_{(k,\ell) \in \cK} \frac1{\norm{x}}
				\Phi^{k\ell}_{ij} \bigl( \norm{x} \, ST_{k\ell}^x(t) \bigr)
			- ST_{ij}^x(t)
		\label{queueing eq 1 scaled}
	\end{equation}
	for all initial states~$x$ with $\norm{x} \geq 1$. Similarly, the scaled 
	idle time process for station~$i$, defined by $SI_i^x (t) := t - 
	\sum_{j=0}^L ST_{ij}^x (t)$, is Lipschitz-1, and the work-conserving 
	property~\eqref{queueing eq 5} implies
	\begin{equation}
		\label{queueing eq 5 scaled}
		\int_0^\infty Q^x_{i0}\bigl( \floor{t} \bigr) \diff I^x_i(t)
		= 0
		= \int_0^\infty SQ^x_{i0}\biggl( \frac{\floor{ \norm{x} t }} 
		{\norm{x}} \biggr) \diff SI_i^x(t).
	\end{equation}

	Now let $\{x'_n\}$ be a sequence of initial states with $\norm{x'_n} \to 
	\infty$ as $n \to \infty$. By the strong law of large numbers, we know 
	that $A_{ij}(n)/n$ converges almost surely to~$p_{ij}$ and 
	$\Phi^{k\ell}_{ij} (n)/n$ converges almost surely to~$P_{k\ell,ij}$ as $n$ 
	tends to~$\infty$ through the integers. So by Lemma~\ref{Lemma: fLLN}, 
	there exists a set~$\Omega' \subset \Omega$ of measure~1 such that, on 
	this set~$\Omega'$, for every~$m\in\N$ and all $(i,j), (k,\ell) \in\cK$,
	\[
		\sup_{t\in[0,m]} \, \Bigl\lvert\,
			\frac1n A_{ij}(nt) - p_{ij}t \,\Bigr\rvert \to 0
		\qquad\text{and}\qquad
		\sup_{t\in[0,m]} \, \Bigl\lvert\,
			\frac1n \Phi^{k\ell}_{ij}(nt) - P_{k\ell,ij}t \,\Bigr\rvert \to0.
	\]
	This shows in particular that on~$\Omega'$, the second term on the right 
	in~\eqref{queueing eq 1 scaled} converges uniformly on any compact 
	set~$[0,m]$ along the sequence~$\{x'_n\}$ to the limit function~$p_{ij}t$. 
	Next, we want to show that the other terms converge uniformly on~$[0,m]$ 
	as well, along an appropriate subsequence.

	As for the first term, since each of the vectors~$SQ^{x'_n} (0)$ lies in 
	the compact set~$[0, 1]^{L^2+L}$, there is a subsequence of~$\{x'_n\}$, 
	that we denote by~$\{x^*_n\}$, along which they converge to some limit 
	vector~$\bQ(0)$. We may assume that $\norm{x^*_n} \geq 1$ for each~$n$, 
	and since $\norm{ SQ^{x^*_n} (0) } = 1$ for each~$n$, it is clear that 
	$\norm{\bQ(0)} = 1$ as well. Now pick any sample point~$\omega \in 
	\Omega'$, and consider the sequence of functions~$\bigl\{ ST^{x^*_n} 
	(\omega, \cdot) \bigr\}$. These functions have Lipschitz-1 components, so 
	if we restrict their domain to the interval~$[0, m]$, it follows from the 
	Arzel\`a--Ascoli theorem~\cite [Theorem~7.3] {billingsley} that the 
	sequence is relatively compact in the space of continuous functions 
	on~$[0,m]$ with the uniform norm. Hence, for every~$m \in \N$, the 
	sequence~$\{x^*_n\}$ has a subsequence~$\{ x^*_{m,n} (\omega) \}$ along 
	which
	\begin{equation}
		\label{Eqn: uoc convergence of T}
		\sup_{t\in[0,m]} \, \normbig{ ST^{x^*_{m,n} (\omega)} (\omega, t) - 
		\bT(\omega, t) } \to 0
		\qquad \text{as $n\to\infty$}
	\end{equation}
	for some limit function~$\bT(\omega, \cdot)$. It then follows from a 
	standard diagonal argument that we can in fact find a subsequence~$\{ 
	x''_n (\omega) \}$ of~$\{x^*_n\}$ such that~\eqref{Eqn: uoc convergence of 
	T}, with $x^\ast_{m,n} (\omega)$ replaced by~$x''_n (\omega)$, holds for 
	all~$m \in \N$ at the same time.

	By \eqref{queueing eq 1 scaled} and Lemma~\ref{Lemma: uoc convergence of 
	sequences}(a), this completes the proof of~\eqref{main thm convergence to 
	fluid limit} and shows that on~$\Omega'$, the fluid limit~$\bQ (\cdot), 
	\bT (\cdot)$ satisfies~\eqref{fluid eq 1}. Moreover, we have seen 
	that~$\norm{ \bQ(0) } =\nobreak 1$. It remains to be shown that the fluid 
	limit satisfies \eqref{fluid eq 2}--\eqref{fluid eq 6} on~$\Omega'$. Since 
	the uniform limit of a sequence of Lipschitz-1 functions is Lipschitz-1, 
	\eqref{fluid eq 2} and~\eqref{fluid eq 3} follow from \eqref{queueing eq 
	2} and~\eqref{queueing eq 3}, and \eqref{fluid eq 6} follows directly from 
	the fact that $Q^x_{ij} (t) \leq 1$ for $i,j \in \{1,\dots,L\}$ and 
	all~$t\geq0$. Clearly, \eqref{fluid eq 4} follows from~\eqref{queueing eq 
	4}. Finally, using the fact that our scaled processes are Lipschitz-1, we 
	see that \eqref{queueing eq 5 scaled} implies
	\[
		0 \leq \int_0^m SQ^{x''_n}_{i0} (t) \diff SI^{x''_n}_i (t)
		= \int_0^m \biggl[ SQ^{x''_n}_{i0} (t) - SQ^{x''_n}_{i0} \biggl( 
		\frac{\floor{ \norm{x''_n}t }} {\norm{x''_n}} \biggr) \biggr] \diff 
		SI^{x''_n}_i (t)
		\leq \frac{m}{\norm{x''_n}}
	\]
	for all~$m\in\N$ because the second integrand is uniformly bounded 
	by~$\norm{x''_n}^{-1}$. By Lemma~\ref{Lemma: uoc convergence of 
	sequences}(b), taking $n\to\infty$ gives $\int_0^m \bQ_{i0}(t) \diff 
	\bI_i(t) = 0$ for all~$m \in \N$, which is equivalent to~\eqref{fluid eq 
	5}.
\end{proof}

\begin{proof}[Proof of Proposition~\ref{Prop: original Markov chain pos rec if 
	fluid limit model is stable}]
	Suppose the fluid model~\eqref{fluid eq 1}--\eqref{fluid eq 6} is stable. 
	By Definition~\ref{Def: fluid model stability}, there exists a constant 
	integer time~$t \geq \delta > 0$ such that any fluid model solution 
	$\bQ(\cdot), \bT(\cdot)$ with $\norm{\bQ(0)} = 1$ satisfies $\bQ(t) = 0$. 
	Now choose a sequence of initial states~$\{x_n\}$ that contains each state 
	in the state space of~$X^2(\cdot)$ exactly once, so that $\norm{x_n} \to 
	\infty$ as~$n \to \infty$. Consider the queue length and service time 
	processes~$Q^{x_n} (\cdot), T^{x_n} (\cdot)$. Let $\Omega'$ be the 
	measure~1 set of sample points from Theorem~\ref{Thm: main thm convergence 
	fluid limit}, and let $\{x'_n\}$ be an arbitrary subsequence of~$\{x_n\}$. 
	Then by Theorem~\ref{Thm: main thm convergence fluid limit}, we can find 
	for each~$\omega \in \Omega'$ a further subsequence~$\{x''_n (\omega)\}$ 
	of~$\{x'_n\}$ and a fluid model solution~$\bQ(\omega, \cdot), \bT(\omega, 
	\cdot)$, so that on the set~$\Omega'$
	\[
		\normbig{ SQ^{x''_n}(t) }
		= \normbig{ SQ^{x''_n}(t) - \bQ(t) } \to 0
		\qquad \text{as $n\to\infty$}.
	\]
	Since we can do this for every subsequence~$\{x'_n\}$, it follows that 
	$\norm{ SQ^{x_n}(t) }$ must converge to~0 almost surely along the original 
	sequence~$\{x_n\}$. But the components of~$SQ^{x_n} (\cdot)$ are 
	Lipschitz-1 and $\norm{ SQ^{x_n} (0) } = 1$, so $\norm{ SQ^{x_n} (t) }$ is 
	bounded by $1 + (L^2+L)t$. Hence, by bounded convergence,
	\[
		\lim_{n\to\infty} \E \normbig{ SQ^{x_n}(t) }
		= \lim_{n\to\infty} \frac1{\norm{x_n}} \,
			\E \normbig{ Q^{x_n} \bigl( \norm{x_n}t \bigr) }
		= 0.
	\]

	Now we recall that the queue length process~$Q^{x_n} (\cdot)$ observed at 
	discrete times has the same law as the Markov chain~$X^2 (\cdot)$ starting 
	from the state~$x_n$. Therefore, we can conclude that there exists an $N 
	\in \N$ such that for each $n>N$, we have that $\norm{x_n} > \norm{x_N}$ 
	and
	\[
		\E_{x_n} \normbig{ X^2\bigl( \norm{x_n}t \bigr) }
		\leq \tfrac12 \norm{x_n}.
	\]
	So if we define the constants $K := \norm{x_N}$, $\gamma := (2t)^{-1}$ and 
	the stopping time $\tau := \normbig{X^2(0)} \, t \bmax 1$, and let 
	$f\colon \cS^2 \to \R_+$ be the function given by $f(x) := \norm{x}$, then 
	the following holds:
	\begin{enumerate}
		\item[(a)] $\E_x\bigl[ f\bigl( X^2(\tau) \bigr) - f(x) \bigr] \leq 
			-\gamma \, \E_x[\tau]$ when $f(x) > K$;
		\item[(b)] The set $F := \{ x \in \cS^2 \colon f(x) \leq K \}$ is 
			finite and $\E_x\bigl[ f\bigl( X^2(1) \bigr) \bigr] < \infty$ for 
			all $x\in F$ (because at most~$L$ customers can arrive to the 
			multiclass network in a single time step).
	\end{enumerate}
	The statement in~(a) is known as a Foster--Lyapunov condition, and 
	together with~(b), it implies positive recurrence of the Markov 
	chain~$X^2(\cdot)$ by Theorem~8.6 in~\cite{robert2003}.
\end{proof}

\end{appendices}

\end{document}